\documentclass[12pt]{amsart}
\usepackage{amsfonts,amssymb,amscd}
\newtheorem{teo}{Theorem}[section]

\newtheorem{prop}[teo]{Proposition}

\newtheorem{lema}[teo]{Lemma}

\newtheorem{obs}[teo]{Remark}

\newtheorem{defnc}[teo]{Definition}

\newtheorem{coro}[teo]{Corollary}

\newcommand{\C}{{\mathbb C}}

\newcommand{\N}{{\mathbb N}}

\newcommand{\R}{{\mathbb R}}

\newcommand{\Z}{{\mathbb Z}}

\newcommand{\dif}{{\rm Diff}^{\omega} ({\mathbb R}, 0)}

\newcommand{\calb}{{\mathcal B}}

\newcommand{\diffM}{{\rm Diff}^{\omega} (M)}

\newcommand{\Leb}{{\rm Leb}}

\newcommand{\leb}{{\rm leb}}

\begin{document}

\title[Ergodic theory of non-discrete actions]{On the ergodic theory of certain non-discrete
actions and topological orbit equivalences}
\author{Julio C. Rebelo}
\date{}

\begin{abstract}
Quasi-invariant measures for non-discrete groups of diffeomorphisms containing a Morse-Smale dynamics are studied. The assumption
concerning the presence of a Morse-Smale dynamics allows us to extend to higher dimensions a number of recently established
results for non-discrete groups acting on the circle. These results are also
applied to show that, for many groups as above, every continuous orbit equivalence must
coincide almost everywhere with a diffeomorphism of the corresponding manifold.
\end{abstract}

\dedicatory{To Professor Yu.S. Il'yashenko, on the occasion of his ${\rm 70}^{\rm th}$ birthday}

\maketitle

\section{Introduction}

Essentially this paper consists of results about the ergodic theory of certain {\it non-discrete}\,
group actions along with a number of related questions, some of them
sort of loosely stated. Motivations for this study stem from a few sources and this Introduction is devoted to
describing some of them by following a chronological order. Statements for the main result of this paper are deferred
to Section~2. Yet, to illustrate the problems we shall be dealing with, let us state an easy ``linear'' version of
these results.

\vspace{0.1cm}

\noindent {\bf Linear Theorem (corollary of Corollary~B and of Theorem~D in Section~2)}. {\sl Consider the standard projective action
of ${\rm PSL}\, (2, \C)$ on the sphere~$S^2$. Suppose that $G_1, \, G_2$ are countable dense subgroups of ${\rm PSL}\, (2, \C)$.
Then we have:
\begin{enumerate}
  \item Suppose that $\mu$ is a $d$-quasiconformal probability measure for $G_1$. Then $\mu$ is absolutely continuous
  and, in particular, $d=2$.

  \item Suppose that $h$ is a topological orbit equivalence between the actions of $G_1$ and $G_2$. Then $h$
  is equivariant and it coincides with an element of ${\rm PSL}\, (2, \C)$.
\end{enumerate}
}

\vspace{0.1cm}

\noindent The conformal character of the projective action of ${\rm PSL}\, (2, \C)$ on $S^2$ is not particularly
important for the above theorem. In fact, similar
statements hold not only to isometry groups of higher dimensional hyperbolic spaces but also to the
projective action of ${\rm PSL}\, (n,\C)$ whose ergodic theory is less developed.

To abridge the discussion in this Introduction, let
${\rm Diff}\, (M)$ denote the group of diffeomorphisms of a compact
manifold $M$ without specifying the regularity: the reader may consider these diffeomorphisms to be real-analytic
though this is not always necessary.
The dynamical study of actions associated to non-discrete representations of (finitely generated)
groups has become more of a focus of interest in the past
two decades. In the so-called ``linear case'', i.e. when the action is simply the restriction of a Lie group
action to a countable subgroup $G$,
it was soon realized that the ``closure'' $\overline{G}$ of $G$ (inside the mentioned Lie group)
contains a wealthy of information on the dynamics of $G$ itself. Whereas the study of
``non-linear dynamics'' is much harder, in certain
cases it is possible to define a dynamically meaningful closure for
a countable non-discrete group $G$ acting on some space. To the best of my knowledge, the first time that these ideas were
effectively put to work was in the papers by Shcherbakov
\cite{arsene1}, \cite{arsene2} naturally
inspired by the seminal work of Y. Il'yashenko \cite{foliations1} on dynamics of complex polynomial vector fields. In his papers,
Shcherbakov deals with the dynamics of non-solvable subgroups of ${\rm Diff}\, (\C,0)$.
Independently and slightly later than Shcherbakov,
the same idea was used by Nakai \cite{nakai} who elegantly established accurate versions of Shcherbakov results.
This collection of results, that deserves to be called
Shcherbakov-Nakai theory, clarifies
the topological dynamics and the basic ergodic theory of the corresponding group actions
(or rather pseudogroup actions). This theory however falls short of settling some questions belonging to the ``higher
ergodic theory'' of these actions, as it will be seen below. Subsequently to their work, Ghys was the first to explicitly
talk about ``non-linear and non-discrete'' group actions
\cite{ghys}. In particular, it is clear from his work that ``non-solvable subgroups of ${\rm Diff}\, (\C,0)$''
are necessarily non-discrete and this provided a kind of explanation for the existence of
Shcherbakov-Nakai theory. It is also shown in \cite{ghys}
that non-discrete finitely generated subgroups of ${\rm Diff}\, (M)$
are quite common and the general idea of defining dynamical closures for the corresponding actions is implicitly advanced.
Basically, it amounts to
constructing non-trivial vector fields on $M$ whose (local) flows can be approximated by actual elements in $G$.
When these vector fields exist, they are said
to belong to the closure of $G$ (cf. Section~2 for an accurate definition).

For a manifold not reduced to a single point, the program proposed by Ghys was satisfactory carried
out on the circle by the author in \cite{rebelo-ENS}. Whereas
no general result is available for higher dimensional manifolds, the method of \cite{rebelo-ENS}
still works in the presence of some ``local contraction'', for example
on a neighborhood of a point $p \in M$ that happens to be a contracting hyperbolic fixed point for some element
$f \in G$. This type of idea was further developed
in \cite{frankandI} (see also \cite{belliart-1}) and used to extend Il'yashenko's theory \cite{foliations1} of
(singular) holomorphic foliations on projective spaces.
Around the same time, and again
independently, Belliart \cite{belliart-2} and the author \cite{rebelo-BBMS} noted that the theory of
\cite{rebelo-ENS} generalizes to non-discrete groups
$G \subset {\rm Diff}\, (M)$ containing a Morse-Smale diffeomorphism (satisfying some very generic
conditions). Though Morse-Smale dynamics does exist in a number of interesting cases,
the corresponding results appear to be much less known than their analogues for
${\rm Diff}\, (\C,0)$ and ${\rm Diff}\, (S^1)$. Hopefully the discussion in this
paper will help to raise further interest in the potential applications of these techniques.

Having briefly mentioned the general method of exploiting the ``closure of a non-discrete subgroup $G$
of ${\rm Diff}\, (M)$'' to study the dynamics of
$G$ itself, let us now turn our attention towards problems belonging to the ``higher ergodic theory'' of $G$.
Here, it may be noted that the usefulness of the above mentioned
vector fields lies primarily in the fact that they allow us
to prove {\it in the non-linear setting}, results that are easy in the linear one. For example, these vector fields
are effective tools to show denseness
of orbits and ergodicity with respect to the Lebesgue measure for the corresponding groups
(the corresponding results being very well-known in the ``linear''
setting). However, there are subtle questions about dynamics that are hard even in the ``linear'' case, and
these include the description of quasi-invariant measures and the
rigidity/flexibility of the orbit structure (in the sense of orbit equivalence, cf. below). These questions are
the prototype of what we have vaguely called the ``higher
ergodic theory'' of group actions. The difficulty of describing quasi-invariant
measures was discussed in \cite{rebelo-PLMS} and it is illustrated by the
following theorem due to Kaimanovich and Le Prince \cite{KPrince}: for every Zariski-dense
finitely generated subgroup $G$ of ${\rm PSL}\, (2, \R)$ there is a non-degenerate
probability measure $\nu$ on $G$ giving rise to a singular stationary measure $\mu$ for the
corresponding (projective) action on $S^1$. Yet, considering the closure of $G$ led us
to new insight in the structure of quasi-invariant measures for non-discrete (and non-solvable)
subgroups of ${\rm Diff}\, (S^1)$. Another motivation for this paper was then to show that most of
the results in \cite{rebelo-PLMS} can be generalized
to higher dimension in the presence of Morse-Smale dynamics.

Finally a last motivation for the discussion conducted in this paper has to do with ``trajectory theory'',
in the spirit of \cite{sinai}, for the corresponding group actions. Among the group actions constructed
in this paper (corresponding essentially to free group actions), we consider whether or not they
are topologically orbit equivalent (topologically OE) with respect to the quasi-invariant Lebesgue measure.
It is then shown that the existence of a topological OE implies that the actions are, indeed, conjugate in $\diffM$.
It may also be pointed out that continuous OE appears to be related to the uniqueness
of certain cross products over $C^{\ast}$-algebras, see \cite{cstaralgebras}. Note that, in the measurable
category, OE bears a close connection with von Neumann algebras and this connection has inspired
a vast literature (see the survey \cite{furman}). In the direction of rigidity results, we shall also sketch
a proof of the fact that an {\it equivariant OE}\, (i.e. a measurable conjugacy) between groups as before
must coincide a.e. with a diffeomorphism of $M$. A few questions about general OE will also be raised.

Let us close this Introduction with a brief outline of the structure of this work. The first thing to be
noted is that this structure is not fully linear: some points are developed to a level of generality greater
than what is needed for the use made in this paper. These situations happen when we want to single out
additional questions or when we believe the corresponding results are interesting in themselves
(in the latter case, see the material in Section~3). This said, Section~2 contains accurate statements
for our main results along with the relevant definitions. Sections~3 and~4 deserve further comments.
As mentioned, vector fields in the ``closure of groups'' constitute the main tool used in our discussion.
Whereas we are mainly concerned with analytic diffeomorphisms, these techniques are also relevant in the
differentiable case, though their extensions beyond the analytic context
hardly been considered (apart from \cite{rebelo-TAMS}). Therefore some was made to clarify the existence
of this type of vector fields in the $C^{\infty}$-case as well. In fact, dealing with groups of
$C^{\infty}$-diffeomorphisms presents us with brand new situations where the notion
of vector fields in the ``closure of groups'' can be brought to shed new light in the dynamics. A remarkable
example of it, being the conclusion that Ghys-Sergiescu \cite{gs} realization of Thompson group in ${\rm Diff}\, (S^1)$
is ``very discrete'' in the sense that it does not satisfy condition~(1) in Section~2. In view of it,
Section~3 contains various results ensuring the existence of vector fields in the ``closure'' of diffeomorphisms
groups (both in the smooth and analytic cases). Proofs are usually harder in the smooth case  and are given in full detail.
Section~4 begins with some extensions of the results in Section~3 for the analytic case. We shall not come back
to the smooth case since we consider that, after the discussion in Section~3, the analogous statements
for $C^{\infty}$-diffeomorphisms can safely be left to the reader. The section then continues with a crucial
result about quasi-invariant measures for groups contained the translations of $\R^n$. Then, after
a brief discussion on ``general'' and ``generic'' (pseudo) Lie algebras associated to groups, our main
results about quasi-invariant measures are derived. Finally Section~5 builds on the previous material
to discuss some rigidity phenomena and a number of related questions.

\vspace{0.2cm}

\noindent {\bf Acknowledgments}. I am very grateful to the editors for the opportunity to submit this
contribution to the celebrations around Y. Ilyashenko's 70 birthday. Special thanks are due to A. Glutsyuk for a few invitations
as well as for many interesting mathematical conversations that we have held in the past few years.

Thanks to D. Fisher for an interesting discussion on orbit-equivalences and, in particular, for having pointed out to
me his work \cite{david}. I am also grateful to C. Connel whose interest in higher dimensional versions of
\cite{rebelo-PLMS} motivated me to write this article.

This paper was written during a visit to the ``unit\'e internationale'' IMPA-CNRS. I would like to thank CNPq/Brazil
for its financial support.

\section{Main results}

Our central statements will involve groups of real analytic diffeomorphisms since essentially all natural examples discussed
in this paper belong to this category. However in Section~3 it will become clear that most of these statements
holds in the $C^{\infty}$-category as well.
Let then $M$ be a compact real analytic manifold and denote by ${\rm Diff}^{\omega} (M)$
the group of its real analytic diffeomorphisms. It may be convenient
to equip ${\rm Diff}^{\omega} (M)$ with the analytic topology of Takens so that
${\rm Diff}^{\omega} (M)$ inherits the Baire property. To define this topology, we first need to fix a {\it complexification}\,
of $M$. For this, choose a realization of $M$ as an (real analytic) sub-manifold of an Euclidean space $\R^s$. Setting
$\R^s \subset \C^s$, a {\it complexification}\, of $M$ is a complex sub-manifold $\widetilde{M} \subset \C^n$ containing $M$
(locally) as $R^s$ is contained in $\C^s$, see \cite{broer} for details. We then denote by $\widetilde{M}_{\tau}$ the set
of points in $\widetilde{M}$ whose distance to $M$ is less than $\tau >0$. The analytic topology is defined by stating
that a sequence of real analytic maps $f_i : M \rightarrow \R$
converges to a (necessarily analytic) map $f_{\infty}$ if, to every $\tau, \, \varepsilon >0$, there corresponds
$i_0$ such that $f_i-f$ has a holomorphic extension $\widetilde{f_i-f}$ to $\widetilde{M}_{\tau}$ satisfying
$\sup_{\widetilde{M}_{\tau}} \Vert \widetilde{f_i-f} \Vert \leq \varepsilon$ whenever $i >i_0$.
By virtue of Cauchy formula, it follows that convergence in this analytic topology implies $C^{\infty}$-convergence.

Next, let us explain Ghys's criterion for non-discrete groups \cite{ghys}. Given a
(finitely generated) group $G \subset \diffM$,
Ghys associates to every (finite) generating set $S$ of $G$ a
sequence of sets $S(j) \subseteq G$ defined as follows: $S(0) =S$ and $S(j+1)$ is the set whose elements are the
commutators written under the form $[F_1^{\pm 1} ,F_2^{\pm 1}]$ where $F_1 \in S(j)$ and
$F_2 \in S(j) \cup S(j-1)$ ($h \in S(0)$ if $j=0$). Then the group $G$ is said to be {\it pseudo-solvable}\,
if for some generating set $S$ as above, the sequence $S(j)$
becomes reduced to the identity for $j$ large enough. Since the definition of pseudo-solvable group may depend
on the generating set, its geometric content is not clear. However, when
$G \subset {\rm Diff}^{\omega} (M)$, then it is known
in some cases that $G$ is pseudo-solvable if and only if it is solvable. Furthermore, whenever
$S$ is constituted by diffeomorphisms sufficiently close to the identity in the analytic topology, it is shown
that every element in
$S (j)$ admits a holomorphic extension to a uniform neighborhood $M_{\tau}$ of $M$ in $\widetilde{M}$. Besides, the
sequence formed by these holomorphic extensions converge uniformly to the identity on $M_{\tau}$. In particular,
these elements form a sequence converging to the identity in the $C^{\infty}$-topology (though they do not converge
to the identity in the above defined analytic topology on $\diffM$).

Going back to pseudo-solvable groups, they are clearly very specific among finitely generated subgroups of ${\rm Diff}^{\omega} (M)$.
Actually, given $k \in N$, $k \geq 2$, consider the product $({\rm Diff}^{\omega} (M))^k$ of $k$-copies of
${\rm Diff}^{\omega} (M)$ endowed with the product analytic topology.
Then, by using the perturbation technique introduced in \cite{broer}, the following can easily be proved
(cf. \cite{rebelo-BBMS}): there is $G_{\delta}$-dense
set $\mathcal{U}_k \subset ({\rm Diff}^{\omega} (M))^k$ such that the diffeomorphisms $f_1, \ldots ,f_k$
generate a free group on $k$ letters provided that the
$k$-tuple $(f_1, \ldots ,f_k)$ belongs to $\mathcal{U}_k$. Combined to Ghys's criterion, this already
shows that ``non-discrete'' subgroups of ${\rm Diff}^{\omega} (M)$ are quite common: their study is hence justified.
However, due to the existence of more subtle examples/applications, for the
purposes of this paper, a finitely generated group $G \subset {\rm Diff}^{\omega} (M)$ will be called
{\it non-discrete}\, if it contains a sequence of elements $h_i$, $h_i \neq {\rm id}$ for every~$i \in \N$,
converging to the identity in, say, the $C^{\infty}$-topology. Naturally this assumption is weaker than
the existence of holomorphic extensions converging to the identity on a uniform domain
$\widetilde{M}_{\tau}$, but it enables us
to obtain results valid also in the smooth category.

Consider a point $q \in M$ and an open neighborhood $U \subset M$ of $q$. Denote by $G_U$
the pseudogroup of diffeomorphisms from open subsets of $U$ to $U$ induced by the restrictions
of in $G$. Then $G$ is said to be conjugate to a finite dimensional Lie group about $q$, if
there is a (germ of) finite dimensional Lie group generating a pseudogroup $\Gamma$
of local transformations on (open subset of) $U$ containing a sub-pseudogroup conjugate to $G_U$ by some (local)
$C^{\infty}$-diffeomorphism $\varphi : U \rightarrow M$ fixing $q$.

A fixed point for $F \in \diffM$ is said to be contracting (resp. expanding) if all the eigenvalues
of $D_pF$ have absolute value less than~$1$ (resp. greater than~$1$).
In the rest of this section, we shall consider finitely generated subgroups of ${\rm Diff}^{\omega} (M)$
satisfying the following conditions:
\begin{enumerate}
\item $G$ is a finitely generated ($C^{\infty}$) non-discrete subgroup of ${\rm Diff}^{\omega} (M)$.

\item $G$ contains a Morse-Smale diffeomorphism $F : M \rightarrow M$. Furthermore every fixed
point $p$ of some power of $F$ that is either contracting or expanding must be also non-resonant.

\item $G$ leaves no proper analytic subset of $M$ invariant.

\item If $U \subset M$ is an open set invariant by $G$, then the action of $G$ on $U$ does not leave
invariant any regular $C^{\infty}$-foliation.

\item $G$ is not conjugate to a finite dimensional Lie group about any of the contracting (or expanding)
fixed point of $F$.

\item Every element $g \neq {\rm id}$ in $G$ has a set of fixed points of codimension at least~$2$.
\end{enumerate}

Concerning the above made assumptions, it is clear that condition~(1) is totally natural since most of our
statements failed to be verified by discrete
groups (cf. below). Condition~(3) is very minor and intended only to simplify the discussion
(note that in dimension~$1$ the action is often asked
to be minimal) and, similarly, Condition~(4) is also natural (analogous to low rank
phenomena) as it serves to rule out a splitting of the dynamics of $G$.
The assumption made in condition~(5) is again harmless since the corresponding statements can easily
be worked out when $G$ is locally conjugate to a finite dimensional Lie group about a contracting/expanding
fixed point of $F$.
The assumption made in condition~(6) plays a role only in the proof of
Theorem~D which uses the results of \cite{david}.
However, it is not known whether or not Condition~(2) is intrinsic to the problems
our just a superfluous assumption needed here for technical reasons (the condition about
non-resonance in fixed points is nonetheless a minor one, cf. Remark~\ref{poincaredulac}).
Yet, extending to higher dimensions the methods of \cite{rebelo-ENS}, for example in the
area-preserving case, remains a major open problem.

Recall that a measure $\mu$ on $M$ is said to be quasi-invariant under $G \subset {\rm Diff}^{\omega} (M)$
if and only if for every element $g \in G$ the Radon-Nikodym derivative $d(g^{\ast}\mu)/d\mu$ exists. Naturally
this happens if and only if for every $g \in G$ the measures $g^{\ast} \mu$ and $\mu$ have the same
null-measure sets. In turn,
an accurate definition of the closure $\overline{G}$ of $G \subset {\rm Diff}^{\omega} (M)$ can be found in Section~3.
With the preceding assumptions, the first main result of \cite{rebelo-PLMS} admits
the following generalization to higher dimensions.

\vspace{0.1cm}

\noindent {\bf Theorem~A}. {\sl Suppose that $G$ satisfies conditions~(1) - (5) above. Let $\mu$ be a
probability measure on $M$ that is quasi-invariant by the closure
$\overline{G}$ of $G$. Then $\mu$ is absolutely continuous.}

\vspace{0.1cm}

Theorem~A has a couple of interesting corollaries whose statements require the following definition.

\begin{defnc}
\label{qcmeasures}
{\rm Let $\mu$ be a probability measure on $M$ and consider a group $G \subset {\rm Diff}^{\omega} (M)$.
Given $d \in \R_+^{\ast}$, the measure
$\mu$ will be called a $d$-quasi-volume for $G$ if there is a constant $C$ such that for every
point $x \in S^1$ and every element $g \in G$, the Radon-Nikodym derivative $d \mu /dg_{\ast} \mu$ satisfies
the estimate
$$
\frac{1}{C} \Vert {\rm Jac}\, [Dg] (x) \Vert^d \leq \frac{d\mu}{dg_{\ast} \mu} (x) \leq C \Vert {\rm Jac}\, [Dg] (x) \Vert^d \, ,
$$
where ${\rm Jac}\, [Dg] (x)$ stands for the Jacobian determinant of $Dg$ at the point~$x$.}
\end{defnc}

In particular, when the action of $G$ on $M$ preserves a conformal structure, the above definition reduces to the usual
notion of $d$-quasiconformal  measures familiar from Patterson-Sullivan theory, see \cite{sullivan}.
The first corollary of Theorem~A (paralleling Corollary~B of \cite{rebelo-PLMS}) states that
groups as in Theorem~A do not possess $d$-quasi-volumes other than those that are absolutely continuous. Note that, compared to previous statements, the two corollaries below are among the very first results valid for this
natural type of quasi-invariant measures beyond the conformal context.

\vspace{0.1cm}

\noindent {\bf Corollary B}. {\sl Suppose that $G$ is as in Theorem~A and let $\mu$ be a $d$-quasi-volume
for $G$. Then $\mu$ is absolutely continuous and $d$ equals the dimension of $M$.}

\vspace{0.1cm}

The previous results can be applied, in particular, to the foliations constructed in \cite{frankandI}. Thus we obtain:

\vspace{0.1cm}

\noindent {\bf Corollary C}. {\sl Every (transverse) $d$-quasi-volume measure for the foliations constructed in
\cite{frankandI} is absolutely continuous.}

\vspace{0.1cm}

Considering the ergodic theory of Kleinian groups and, in particular, the prominent role played by Patterson-Sullivan
measures, Corollary~B seems to indicate that the ergodic theory of non-discrete groups may share properties with
(geometrically finite) convex co-compact Kleinian groups. Similarly, in analogy with rational maps,
the Lyubich-Minsky laminations associated to the latter may be considered, cf. \cite{MLyubich}. These laminations
carry several types of natural measures, including Patterson-Sullivan type measures, cf. \cite{KLyubich}. It is then
natural to ask if (or when) the holonomy pseudogroup of the laminations in question is discrete. In fact, the
pseudogroup of these laminations seems to provide
a suitable framework to ask whether or not the ``pseudogroup arising from the several inverse branches
of a rational map is discrete''. The understanding of discrete/non-discrete situations, along with the ergodic
consequences immediately derived in the latter case, would constitute a fine addition to Sullivan's dictionary,
besides shedding additional light in the structure of the measures constructed in \cite{KLyubich}.

Going back to Theorem~A, let us also mention that is probably not hard to establish a comparison theorem between
quasi-invariant measures with Hausdorff dimension~$d$
and the corresponding $d$-dimensional Hausdorff measures analogous to the statement provided in \cite{rebelo-PLMS}.
The sequence of this article, however, goes
in the direction of rigidity for orbit equivalences where there is a few new questions.

Concerning orbit equivalences and rigidity results, let $G_1, G_2 \subset {\rm Diff}^{\omega} (M)$ be two
finitely generated groups. Suppose also that
$G_1$ is contained in a finite dimensional Lie group effectively acting on $M$ (or more generally contained
in a locally compact subset of ${\rm Diff}^{\omega} (M)$). In other words, the action of $G_1$ is ``linear''
according to our previous terminology.
Recall that the groups $G_1, G_2$ are said to be {\it continuously orbit equivalent}\, if there is a homeomorphism
$h$ of $M$ taking orbits of $G_1$ to orbits of $G_2$.
Note, in particular, that this definition does not require $h$ to be equivariant and, in particular, it does not
imply that $G_1, \, G_2$ are isomorphic. With these definitions,
one of our main results can be stated as follows.

\vspace{0.1cm}

\noindent {\bf Theorem~D}. {\sl Suppose that $G_1, \, G_2 \subset {\rm Diff}^{\omega} (M)$, where the action
of $G_1$ is supposed to be ``linear''. Suppose also that $G_2$ fulfils all the
conditions~(1) - (4) as well as condition~(6) in
Section~2. Then every homeomorphism $h: M \rightarrow M$ realizing an orbit equivalence
between the actions of $G_1, \, G_2$ on $M$ is equivariant and coincides with a real analytic diffeomorphism
of $M$}.

\vspace{0.1cm}

Theorem~D will be proved in Section~5 which includes a comparison between Theorem~D and some previous
related results. Being rather flexible, our methods yield further results whose proofs will
only be sketched there to keep the length of the article under control. In particular, it allows us to conclude that
every (measurable) {\it equivariant}\, orbit equivalence between $G_1, \, G_2$ as above coincides a.e. with an element of
${\rm Diff}^{\omega} (M)$. The question on whether
or not there exist non-equivariant orbit equivalences between these actions is very interesting and some implications
will be mentioned at the end of Section~5.
Another question that will only be mentioned in Section~5 has to do
with the ``linear'' assumption imposed on $G_1$. If this condition is dropped, the result hinges on deciding
whether or not a curious phenomenon of convergence involving the group $G_1$ can occur.

\section{Closure of groups and vector fields}

Albeit some of the material in this section is well-known \cite{frankandI},
\cite{belliart-1}, \cite{belliart-2} and \cite{rebelo-BBMS}, our discussion is mostly original.
In fact, a common trace in the mentioned papers is that they deal with
``analytic'' limits contained in the closure of $G$ while we consider smooth limits making sense
for $C^{\infty}$-groups (a brief discussion of smooth limits appeared previously in \cite{rebelo-TAMS}).
Actually the discussion conducted below involves various topologies
and their corresponding notions of limits so as to make clear what is valid beyond the ``analytic'' condition.
As a consequence, this section is more comprehensive than what is needed
later. Nonetheless I believe that clarifying these issues about topologies makes this discussion worth of interest.

In the sequel $M$ is a manifold and $G$ is a finitely generated group of diffeomorphisms of $M$. In terms of
regularity, it suffices to assume these diffeomorphisms
to be of class $C^{\infty}$ or real analytic: the discussion is more concerned with the notions of convergence
than with the regularity of the diffeomorphisms themselves.
Whereas the case $C^r$ is also sufficient at certain points,
the corresponding adaptations are going to be left to the reader.

To begin with, let $U \subseteq M$ be a connected open set and consider the restriction to $U$ of elements in $G$.
Consider also a vector field
$X$ defined on $U$ and denote by $\phi^t$ its local flow. The vector field $X$ is said to be {\it in the $C^{\infty}$-closure of
$G$ (relative to $U$)}\, if the following holds:
\begin{description}

\item[({\sl a})] $X$ is a $C^{\infty}$-vector field.

\item[({\sl b})] For every $t \in \R$ and $V \subset U$ such that $\phi^s$
is defined on $V$ for all $s \in [0,t]$, there exists a sequence
of elements $g_i \in G$ whose restrictions to $V$ converge in the $C^{\infty}$-topology to the diffeomorphism
$\phi^t : V \rightarrow \phi^t ( V)$.
\end{description}
Analogously, the vector field is said to be in the $C^k$-closure of $G$
if the sequence $g_i$ approximates $\phi^t$ only in the $C^k$-topology.

The analytic variant of the closure of $G$ is however slightly more subtle.
Fix a {\it real analytic}\, vector field $X$ defined on $U$ as above.
Then consider a neighborhood $\widetilde{U}$ of $U$ in $\widetilde{M}$ where $X$ admits a holomorphic
extension denoted by $\widetilde{X}$ whose (real) local flow will still be denoted by $\phi^t$.
Then $X$ is said to belong to the $C^{\omega}$-closure of $G \subset \diffM$
if $\widetilde{U}$ can be chosen so as to satisfy the following condition:
\begin{description}

\item[({\sl c})] For every $t \in \R$ and $\widetilde{V} \subset \widetilde{U}$
such that $\phi^s$ is defined on $\widetilde{V}$ for all $s \in [0,t]$, there exists a sequence
of elements $g_i \in G$ possessing holomorphic extensions $\tilde{g}_i$ to $\widetilde{V}$ and such that
these extensions $\tilde{g}_i$ converge uniformly on $\widetilde{V}$ to $\phi^t : \widetilde{V} \rightarrow
\phi^t (\widetilde{V}) \subset \widetilde{M}$.
\end{description}
To abridge notations, the sequence $\{ g_i \} \subset G$ is said to $C^{\omega}$-approximate $\phi^t$ when the
condition~({\sl c}) is verified (the phrase $X$ is in the $C^{\omega}$-closure of $G$ will also be used).
Similarly, when a statement is valid for all $C^r, \, C^{\infty}$ and $C^{\omega}$ closures of a group
$G$, then we shall simply talk about the closure of $G$ which will be denoted by $\overline{G}$.

Obviously the notion of belonging to the ($C^r, \, C^{\infty}, \,C^{\omega}$) closure of $G$ is relative to the
domain of definition of the vector field in question: this will implicitly be understood whenever
no misunderstanding is possible. Let us also point out that a similar notion can be defined for diffeomorphisms:
a diffeomorphism $\phi : U \rightarrow \phi (U) \subset M$
is said to be in the ($C^r, \, C^{\infty}$) closure of $G$ if it can be
approximated in the ($C^r, \, C^{\infty}$) topology by the restriction to $U$ of actual elements of $G$.
To be in the $C^{\omega}$-closure of $G$, the diffeomorphism $\phi : U \rightarrow \phi (U) \subset M$
needs to have a holomorphic extension to a neighborhood $\widetilde{U}
\subset \widetilde{M}$ of its domain of definition where it is, in addition, a uniform limit of
(holomorphic extensions of) elements in $G$.

It is clear from the above definition that vector fields in
$\overline{G}$ form a {\it pseudo-Lie algebra}\, in the sense that new vector fields obtained out of
vector fields in $\overline{G}$ by means of linear combinations (with constant coefficients) and
of Lie brackets still belong to the closure of $G$ relative to the open set where
they are defined.
Furthermore it is equally clear that this pseudo-Lie algebra is closed: if $\{ X_i \}
\subset \overline{G}$ is a sequence of vector fields defined on some common domain $U$ and
converging (in the corresponding topology) to a vector field $X_{\infty}$, then $X_{\infty}$ is automatically
contained in $\overline{G}$. Finally, $G$ acts on its
pseudo-Lie algebra by pull-backs. The pull-back of a vector field $X \in \overline{G}$ relative to the domain
of definition $U_X$ of $X$ by a diffeomorphism $f$ in $G$ lies in $\overline{G}$ relative to
its domain de definition $f^{-1} (U_X)$. All these remarks apply to local diffeomorphisms in $\overline{G}$ as well.

Vector fields in the closure of $G$ constitute a powerful tool to access the dynamics of $G$. Naturally,
when $G$ is contained in some (finite dimensional) Lie group acting on $M$, non-trivial vector fields in the
closure of $G$ that, in addition, are globally defined on $M$, always exist. In fact, being non-discrete, the closure of
$\overline{G}$ inside the Lie group in question is itself a Lie group with non-trivial Lie algebra. The mentioned vector fields
are therefore nothing but the image of vector in the Lie algebra in question by the induced representation in the space
of vector fields on $M$. This situation where $G$ is contained in a Lie group acting on $M$ will often be
called ``linear'' as opposed to the ``non-linear'' situation where $G$ is simply a finitely generated subgroup
of $\diffM$. In the ``non-linear'' setting, however, the existence of these vector fields is in general hard to establish.

Let us begin by pointing out the type of information required to construct non-identically zero
vector fields in the closure of a group. For this, consider
an open set $U \subset M$ which may
be identified to a ball about the origin in $\R^n$.
Suppose that a group $G$ of diffeomorphisms of $M$ contains a sequence of elements $\{ h_i \}$
verifying the following conditions.
  \begin{itemize}
    \item The restrictions to $U$ of the diffeomorphisms $h_i$ form a sequence $\{ h_{i \vert U} \}$ of
  one-to-one maps converging $C^{\infty}$ to the identity. Besides $h_{i \vert U} \neq {\rm id}$ for every $i \in \N$.

    \item For some $r \geq 1$ in $\N$, there is a uniform constant ${\rm Const}$ such that
    $$
    \Vert h_i -{\rm id} \Vert_{r, U} < {\rm Const} \Vert h_i -{\rm id} \Vert_{0, U}
    $$
    where, for $k \in \N$, the bars $\Vert \, . \, \Vert_{k, U}$ stand for the $C^k$-norm on $U$.
  \end{itemize}
Then we have:

\begin{prop}
\label{vectorfield1}
Assume that $G$ contains a sequence of elements satisfying the above conditions for some $r \geq 1$.
Then there exists a non-identically zero vector field $X$ defined on $U$ and contained in the $C^{r-1}$-closure of $G$.
\end{prop}

\begin{proof}
For every $i \in \N$, let $C_i$ denote the $C^0$-norm of $h_i-{\rm id}$ on $U$, i.e. $C_i = \Vert h_i -{\rm id} \Vert_{0, U}$.
Then consider the vector field $X_i$ defined on $U$ by
\begin{equation}
X_i (x) = \frac{1}{C_i} \cdot (h_i (x) -x) \, ,
\label{thefield}
\end{equation}
where $x = (x_1, \ldots , x_n)$ and where $h_i (x) -x$ stands for the vector going from the point $x$ to the point
$h_i(x)$.

Next note that the $C^{r}$-norm on $U$ of the vector fields $X_i$ is
uniformly bounded by ${\rm Const}$. Thus, by applying Ascoli-Arzela
theorem, we can suppose without loss of generality that the
sequence of vector fields $X_i$ converge in the $C^{r-1}$-topology to
a $C^{r-1}$ vector field $X_{\infty}$ defined on $U$. The definition
of $C_i$ immediately implies that $X_{\infty}$ is not
identically {\it zero} since $\Vert X_i(x) \Vert_{0, U} =1$ for every $i \in \N$.

Finally let us prove that the local flow $\phi^t$ of
$X_{\infty}$ can be $C^{r-1}$-approximated by (suitable restrictions of) elements in $G$.
For this, consider an open set $U_0$ compactly contained in $U$
and $t_0 \in \R$ so that $\phi^t : U_0 \rightarrow U \subset \R^n$
is defined whenever $0 \leq t \leq t_0$. It is then a straightforward application of Euler's
polygonal method to check that the (restriction to $U_0$) of the elements in
$G$ given by $h_i^{[tt_0 /C_i]}$ converges on $U_0$ to $\phi^t_X$ in the $C^{r-1}$-topology
(where the brackets $[.]$ stand for the integral part).
The proof of the proposition is over.
\end{proof}

Note that the statement above {\it cannot}\, directly be adapted to handle the $C^{\infty}$-case. An analytic
version of it, however, can easily be obtained.

\begin{prop}
\label{vectorfield1star}
Assume that $G \subset \diffM$ contains a sequence $\{ h_i\}$ of elements satisfying the following conditions:
\begin{itemize}
  \item The elements $\{ h_i\}$ possess holomorphic extensions $\{ \tilde{h}_i\}$ to a (fixed) neighborhood $\widetilde{U}
  \subset \widetilde{M}$ of $U \subset M \subset \widetilde{M}$. Moreover $h_{i} \neq {\rm id}$ for every $i \in \N$.

  \item Denoting by $\Vert . \Vert_{\widetilde{U}}$ the $C^0$-norm of a function defined on $\widetilde{U}$,
  the sequence $\{ h_i\}$ is such that $\Vert \tilde{h}_i -{\rm id} \Vert_{\widetilde{U}}$ converges to zero as $i \rightarrow
  \infty$.

  \item There exists $V \subset U$ and a neighborhood $\widetilde{V} \subset \widetilde{M}$ compactly contained
  in $\widetilde{U}$ along with a uniform constant~$C$ such that
  $$
    \Vert \tilde{h}_i -{\rm id} \Vert_{\widetilde{U}} < {\rm Const} \Vert \tilde{h}_i -{\rm id} \Vert_{\widetilde{V}} \, ,
  $$
  for every $i \in \N$.
\end{itemize}
Then there is a non-identically zero analytic vector field $X$ defined on $V$ and contained in the $C^{\omega}$-closure
of $G$.
\end{prop}

\begin{proof}
The proof is an immediate adaptation of the proof of Proposition~\ref{vectorfield1}. It suffices to replace
Ascoli-Arzela theorem by Montel theorem.
\end{proof}

Before continuing our discussion, let us state a useful elementary lemma.
Fixed $d \in \N$, denote by ${\rm Pol}\, (d,n)$
the space of maps from $\R^n \rightarrow \R^n$ whose components are
polynomials of degree at most $d$ in the multi-variable $x = (x_1, \ldots ,x_n)$.

\begin{lema}
\label{vectorfield4}
Let $\{ P_j \} \subset {\rm Pol}\, (d,n)$ be a sequence of polynomials converging to the
identity, with $P_j \neq {\rm id}$ for every $j \in \N$. Let $V \subset \R^n$ be an open set
and suppose we are given $r \in \N$. Then there is a constant $C=C(d,n, V,r)$ such that
\begin{equation}
\Vert P_j -{\rm id} \Vert_{r, V} < C_r \Vert P_j -{\rm id} \Vert_{0, V} \, \label{pol1}
\end{equation}
uniformly on~$j$. Similarly if $\widetilde{V} \in \C^n$ is a compact neighborhood of $V$ and
we consider the polynomials $P_j$ as defined on $\C^n$, then there is a constant $C$ such that
\begin{equation}
\Vert P_j -{\rm id} \Vert_{\widetilde{V}} < C_r \Vert P_j -{\rm id} \Vert_{0, V} \, \label{pol2}
\end{equation}
uniformly on~$j$.
\end{lema}

\begin{proof}
By virtue of Cauchy formula, the second assertion in the statement implies the first one. Thus it suffices
to prove Estimate~(\ref{pol2}). Since $\{ P_j \}$ has bounded degree and it converges to the identity
in the $C^0$-topology over the open
set $U$, it is clear that the coefficients of $P_j$ converge to the corresponding
coefficients of the identity map $x = (x_1, \ldots ,x_n) \mapsto (x_1, \ldots ,x_n)$.

Suppose that a small $\delta >0$ is fixed. Modulo
multiplying each $P_j-{\rm id}$ by some constant $c(j) \in \R$, we may assume that, for every $j$, all the coefficients
of $P_j -{\rm id}$ have absolute value bounded by~$\delta$ and, in addition, that at least one of these coefficients has
absolute value exactly equal to~$\delta$. Furthermore, if $\delta >0$ is chosen sufficiently small, then
the norm $\Vert P_j-{\rm id} \Vert_{\widetilde{V}}$ is still bounded by~$1$ (recall that $\widetilde{V}$ is compact).
Now suppose for a contradiction that the statement is false. Since $\Vert P_j -{\rm id} \Vert_{\widetilde{V}} \leq 1$,
modulo passing to a subsequence, we can suppose that $\Vert P_j -{\rm id} \Vert_{0, V} \rightarrow 0$, otherwise there
is nothing to be proved. Nonetheless the space
${\rm Pol}\, (d,n)$ is locally compact since it is modeled by some finite dimensional Euclidean space.
Furthermore all the coefficients of $P_j -{\rm id}$ have absolute value bounded by~$\delta$ so that,
again modulo passing to a subsequence, we can assume
that the sequence $\{ P_j -{\rm id} \}$ converges towards a polynomial map $Q_{\infty}$. The polynomial
$Q_{\infty}$ is not identically zero since at least one of its coefficients has absolute value equal to~$\delta >0$.
A contradiction then arises from observing that $\Vert Q_{\infty} \Vert_{0,V}$ must be zero since it is the limit
of the sequence $\{ \Vert P_j -{\rm id} \Vert_{0, V} \}$ which converges to zero. The lemma is proved.
\end{proof}

For the rest of this section, we shall exclusively deal with groups of diffeomorphisms of $M$ that are smooth or
real analytic, the more technical adaptations related to the $C^r$-case will be left to the reader.

Consider a (smooth or real analytic)
diffeomorphism $F$ of $M$ fixing a point $p \in U$. In local coordinates, we can think of $F$
as a local diffeomorphism of $\R^n$ fixing the origin. Suppose that the point $p$ (identified to the origin)
is a contracting fixed point of $F$.
Consider also a sequence of elements $\{ g_i \} \subset G$ whose
restrictions $g_{i,U}$ to $U$ are different from the identity for every~$i \in \N$. The sequence
$\{ g_{i,U} \}$ is supposed to satisfy {\it one of the following assumptions}.
\begin{description}
  \item[{\rm SC}] The sequence $\{ g_{i,U} \}$ converges to the identity on $U$ in the $C^{\infty}$-topology.
  \item[{\rm HC}] There is a neighborhood $\widetilde{U} \subset \widetilde{M}$ of $U$ where
  all the maps $g_{i,U}$ have a holomorphic extension denoted by $\tilde{g}_{i,U}$. Moreover the sequence
  $\{ \tilde{g}_{i,U} \}$ converges uniformly to the identity on $\widetilde{U}$.
\end{description}

The rest of this section is devoted to the proof of a couple of propositions, namely:

\begin{prop}
\label{vectorfield3}
For $U$ as above, suppose that the group $G$ contains an element $F$ which has a non-resonant
contracting fixed point $p \in U$. Suppose, in addition, that $G$ contains a sequence of elements
$\{ g_i \}$ satisfying condition~(SC). Then, given $r \in \N$ and modulo reducing $U$, there exists
a non-identically zero vector field $X$ defined on $U$ and contained in the
$C^{r-1}$-closure of $G$. Besides $X$ has the same regularity as $F$, i.e. $X$ is smooth/analytic provided
that $F$ is so.
\end{prop}

This proposition admits the following ``analytic'' analogue:

\begin{prop}
\label{vectorfield3star}
For $U$ as above, suppose that the group $G$ contains an element $F$ which has a non-resonant
contracting fixed point $p \in U$. Suppose, in addition, that $G$ contains a sequence of elements
$\{ g_i \}$ satisfying condition~(HC). Then, modulo reducing $U$, there exists
a non-identically zero real analytic $X$ defined on $U$ and contained in the $C^{\omega}$-closure of $G$.
\end{prop}

\begin{obs}
\label{poincaredulac}
{\rm Before approaching the proofs of these propositions, let us point out that the assumption
concerning the non-resonant character of the contracting fixed point of $F$ is not indispensable
for the corresponding statements to hold. To substantiate this claim, observe that
a contracting fixed point can have only finitely many resonance modes. Thus, as long as $F$ is smooth
(resp. analytic), $F$ is smoothly (analytically) conjugate
to a polynomial normal form, called Poincar\'e-Dulac normal form, cf. \cite{EMS1}. The use of this
Poincar\'e-Dulac normal form enables one to check that the fundamental estimates used below,
namely Estimate~(\ref{controllingtail}) and Estimate~(\ref{justabove}), still hold in the corresponding context.
In other words, the assumption about non-resonance is only intended to simplify the calculations since
$F$ becomes a diagonal matrix in appropriate coordinates.}
\end{obs}

To prove Proposition~\ref{vectorfield3}, the central point is to construct a sequence $\{h_i \} \subset G$
fulfilling the conditions of Proposition~\ref{vectorfield1}, for every {\it a priori}\, fixed~$r \in \N$. Similarly
the sequence $\{h_i \} \subset G$ is supposed to satisfy the conditions of Proposition~\ref{vectorfield1star}
in the case of Proposition~\ref{vectorfield3star}. Let us first work out the proof of Proposition~\ref{vectorfield3}.

Let $U$ be identified to the unit ball of $\R^n$ and let $p$ be identified to the origin.
Consider the sequence $\{ g_{i,U} \} \subset G$, $g_{i,U} \neq {\rm id}$ for every~$i \in \N$,
converging $C^{\infty}$ to the identity and obtained as restrictions of elements in $G$.
To abridge notations, this sequence will simply be denoted by $\{ g_i\}$
since no misunderstanding is possible.
Denote by $P_i^r (x)$ the Taylor polynomial $P_i^r (x) = g_i (0) + g^{(1)}_i (0) x +
\cdots + g_i^{(r)} (0) x^r / r ! \,$. We have
\begin{equation}
\Vert g_i (x) - P_i^r (x) \Vert  \leq \frac{1}{(r+1)!}
\sup_{x \in U} \Vert g_i^{(r+1)} (x) \Vert \, . \,  \Vert x \Vert^{r+1} \, . \label{Taylor}
\end{equation}
Since $\{ g_i \}$ converges $C^{\infty}$ to the identity, it follows that
$\sup_{x \in U} \Vert g_i^{(k)} (x) \Vert$
converges to {\it zero} for $k \neq 1$ while, for $k =1$, $\sup_{x \in U} \Vert g_i^{(1)} (x)\Vert$ converges
to the identity. If the sequence $\{ g_i \}$ satisfies the assumptions of Proposition~\ref{vectorfield1}, then there is
a non-identically zero vector field in the $C^{r-1}$-closure of $G$, though it will not be clear that $X$ can be
chosen as a real analytic vector field. In general, we proceed as follows.

\begin{proof}[Proof of Proposition~\ref{vectorfield3}]
Recall that $U$ is identified to a neighborhood of the origin in $\R^n$. Modulo working over $\C$,
$F$ becomes identified to the homothety $(x_1, \ldots ,x_n) \longmapsto
(\lambda_1 x_1 , \ldots ,\lambda_n x_n)$ where the eigenvalues $\lambda_i$ of $D_0F$
verify $\sup_{i=1, \ldots ,n} \vert \lambda_i \vert <1$ (alternatively it is possible to work with
the ``block-diagonal'' form of $D_0F$ involving both real and purely imaginary parts of the eigenvalues).

The next step consists of constructing a sequence $\{ h_i \} \subset G$ verifying the assumptions of
Proposition~\ref{vectorfield1}. The sequence $\{ h_i \}$ will be constructed by means of the sequence
$\{ g_i \}$ by setting $h_i = F^{-k(i)} \circ g_i \circ F^{k(i)}$ for a suitable sequence $\{ k(i) \}
\subset \N$ to be defined below. To prove that a suitable choice of integers $k(i)$ leads to
a sequence $h_i$ verifying the indicated conditions, let us first permute the coordinates so as to have
$\vert \lambda_1 \vert < \cdots < \vert \lambda_n \vert < 1$. In particular, the subset of $\R$ consisting of the absolute values
of all possible combinations having the form $\lambda_1^{\alpha_1}
\ldots \lambda_n^{\alpha_n}/\lambda_l$, $\alpha_1, \ldots ,\alpha_n \in \N$ and $l \in
\{ 1, \ldots ,n\}$ admits a supremum. Besides this supremum
is attained by finitely many elements.

Choose $N$ large enough to have
$\vert \lambda_n \vert^N < \vert \lambda_1 \vert$. Without loss of generality,
we can assume that $r > N$. Finally, setting $\Lambda^s.x = (\lambda_1^s x_1, \ldots ,\lambda_n^s x_n)$, it follows that
$$
h_{i,k}  =  [D_0 F]^{-k} . g_i \circ F^{k}
=  [D_0 F]^{-k}.[\sum_{j=0}^r g^{(j)}_i (0) \Lambda^{k r} .x + R_{i,k } (x)]
$$
for every $k \in \N$, where
\begin{equation}
[D_0 F]^{-k}. R_{i,k} (x) \leq \frac{1}{(r+1)!}[D_0 F]^{-k}
\sup_{x \in U} \Vert g_i^{(r+1)} (x) \Vert . \Vert \Lambda^{k(r+1)} .x
\Vert < {\rm Const} \sup_{x \in U} \Vert g_i^{(r+1)} (x) \Vert \, , \label{forTaylorremainder}
\end{equation}
for a uniform constant ${\rm Const}$ (recall that $\vert \lambda_n \vert^N < \vert \lambda_1 \vert$).
Now, note that $\sup_{x \in U} \Vert g_i^{(r+1)} (x) \Vert$ converges to zero with $i$ since $g_i$
converges in the $C^{\infty}$-topology to the identity. By using, for example, the integral representation
for the Taylor remainder term, it is straightforward to obtain similar estimates for derivatives so as to ensure
that the $C^r$-norm of $[D_0 F]^{-k}. R_{i,k} (x)$ on $U$ converges to zero with~$i$ and
independently of $k$. Fix then a decreasing sequence $\{ \delta_i \}$ converging to $0$ such that
\begin{equation}
\Vert [D_0 F]^{-k}. R_{i,k} (x) \Vert_{r, U} < \delta_i \label{controllingtail}
\end{equation}
for every $i, k \in \N$. Next, consider the sequence of polynomial maps $P_{i,k}$ given by $P_{i,k} = [D_0 F]^{-k}.
[\sum_{j=0}^r g^{(j)}_i (0) \Lambda^{k r} .x]$. These polynomials are elements of ${\rm Pol}\, (r,n)$ converging to
the identity map so that Lemma~\ref{vectorfield4} yields a constant $C$ verifying
\begin{equation}
\Vert P_{i,k} -{\rm id} \Vert_{r, U} < C \Vert P_{i,k} -{\rm id} \Vert_{0, U} \, . \label{justabove}
\end{equation}

For the convenience of the reader, the argument will now be split into two cases.

\noindent {\it Case 1}. Modulo passing to a subsequence, suppose that $g_i (0) \neq 0$ for every~$i$.

In this case, for every $i$ fixed, there is a smallest positive integer
$K_i$ such that
\begin{equation}
\sup_{x \in U} \Vert [D_0 F]^{-k(i)} .[\sum_{j=0}^r g^{(j)}_i (0) \Lambda^{k(i)r} .x -{\rm id}] \Vert > \max\{ 10 \delta_i ,
10 C \delta_i \}
\label{controllingpolynomial1}
\end{equation}
where $C$ is the constant involved in Estimate~(\ref{justabove}).
The existence of $K_i$ follows easily from the fact that $\Vert g_i (0)-0 \Vert >0$ so that
$\lim_{k\rightarrow \infty} \Vert [D_0 F]^{-k} .(g_i(0)-0) \Vert = \infty$.
Now set $k (i) =K_i$ so as to define $h_i = h_{i,k(i)} = h_{i,K_i}$. We need to check that the
resulting sequence $h_i = F^{-k(i)} \circ g_i \circ F^{k(i)}$ satisfies the assumptions of
Proposition~\ref{vectorfield1}. The first step is to show that this sequence converges to the identity
(on $U$) in the $C^{\infty}$-topology. This goes as follows. Recall that
$$
h_i = [D_0 F]^{-k (i)} . g_i \circ F^{k (i)}
= [D_0 F]^{-k (i)}.[\sum_{j=0}^r g^{(j)}_i (0) \Lambda^{k (i) r} .x + R_{i,k(i) } (x)] \, .
$$
The function $[D_0 F]^{-k (i)}.R_{i,k(i)}$ clearly converges in the $C^{\infty}$-topology to the identically zero function
as it follows from using, for example, the integral form of Taylor remainder term (the $C^0$-norm was explicitly
estimated in Formula~(\ref{forTaylorremainder})). Thus we only need to show that
the truncation $J^r h_i$ given by
$$
J^rh_i= [D_0 F]^{-k(i)} .[\sum_{j=0}^r g^{(j)}_i (0) \Lambda^{k(i)r} .x]
$$
converges to the identity in the $C^{\infty}$-topology. For this recall that $k (i)$ was defined
as the smallest positive integer verifying Estimate~(\ref{controllingpolynomial1}). In turn, this implies
that
$$
\sup_{x \in U} \Vert [D_0 F]^{-k(i)-1} .[\sum_{j=0}^r g^{(j)}_i (0) \Lambda^{(k(i)-1)r} .x - {\rm id}] \Vert \leq
\max\{ 10 \delta_i , 10 C \delta_i \}  \, .
$$
In particular, $\sup_{x \in U} \Vert J^rh_i -{\rm id} \Vert$ is bounded by $\lambda_1^{-1} \max\{ 10 \delta_i , 10 C \delta_i \}$.
Since $\delta_i \rightarrow 0$, it follows that $\{ \Vert J^rh_i -{\rm id}\Vert_{0,r} \}$ converges to zero as well. By virtue of
Lemma~\ref{vectorfield4}, we conclude that $\{ \Vert J^rh_i -{\rm id}\Vert_{0,r}\}$ converges
to zero in the $C^k$-topology for every $k$ what is equivalent to saying that it converges to zero in the
$C^{\infty}$-topology.

Finally to ensure that the second condition of Proposition~\ref{vectorfield1} is also verified, we proceed as
follows. First note that we have
\begin{eqnarray}
\nonumber \Vert h_i -{\rm id} \Vert_{0,U} & = & \Vert J^rh_i - {\rm id} + [D_0 F]^{-k (i)}.R_{i,k(i) } (x) \Vert_{0,U} \\ \nonumber
& \geq & \Vert J^rh_i - {\rm id} \Vert_{0,U} - \Vert [D_0 F]^{-k (i)}.R_{i,k(i) } (x) \Vert_{0,U} \\
& \geq & \max\{ 9 \delta_i , 9 C \delta_i \} \, , \label{estimateoneside}
\end{eqnarray}
where estimates~(\ref{controllingpolynomial1}) and~(\ref{controllingtail}) were used.

On the other hand, we also have
\begin{eqnarray}
\nonumber \Vert h_i -{\rm id} \Vert_{r,U} & \leq & \Vert J^rh_i - {\rm id} \Vert_{r,U} + \Vert [D_0 F]^{-k (i)}.R_{i,k(i) } (x) \Vert_{r,U}
\\ \nonumber
& \leq & C \Vert J^rh_i - {\rm id} \Vert_{0,U} +\delta_i \\
& \leq &  C \max\{ 10 \delta_i , C \delta_i \} +\delta_i \, , \label{estimateotherside}
\end{eqnarray}
where estimates~(\ref{controllingpolynomial1}) and~(\ref{justabove}) were used. The statement now follows from
comparing Estimates~(\ref{estimateoneside}) and~(\ref{estimateotherside}).

Summarizing, it was shown above the existence of a sequence $\{ h_i \} \subset G$ satisfying the conditions
of Proposition~\ref{vectorfield1}. Therefore there exists a non-identically
zero vector field $X$ of class $C^{r-1}$ lying in the $C^{r-1}$-closure of $G$. It remains to show that, in fact, $X$
can be chosen smooth/analytic according to the regularity of $F$. However, to check this assertion, it suffices
to revisit the sequence $\{ h_i \}$ constructed above. Indeed, it follows from the previous construction that
$h_i$ can be written as $h_i = \overline{h}_i + R_i$ where $\overline{h}_i, \, R_i$ satisfy the following conditions.
\begin{itemize}
  \item $\overline{h}_i$ is a polynomial whose degree is uniformly bounded (independently of $i$). Moreover all the
  coefficients of $\overline{h}_i -{\rm id}$ have absolute value bounded by $\delta >0$.
  \item At least one of the coefficients of $\overline{h}_i -{\rm id}$ has absolute value exactly equal to~$\delta$.
  \item The sequence of quotients defined by $\Vert R_i \Vert_{r,U} / \Vert \overline{h}_i \Vert_{0,U}$ converges to zero.
\end{itemize}
Now recall that $X$ is the $C^{r-1}$-limit (and hence the $C^0$-limit)
of the sequence of vector fields $X_i$ defined in Equation~(\ref{thefield}).
However, the fact that the quotients $\Vert R_i \Vert_{r,U} / \Vert \overline{h}_i \Vert_{0,U}$ converge to zero implies that
the sequence $X_i$ has the same $C^{0}$-limit as the sequence $Y_i$ given by
$$
Y_i (x) = \frac{1}{\Vert \overline{h}_i -{\rm id} \Vert_{0, U}} \cdot (\overline{h}_i (x) -x) \, .
$$
It is clear that the $C^0$-limit of $\{ Y_i\}$ is a polynomial vector field and, the $C^0$-limit being unique,
it must coincide with $X$. Therefore $X$ is actually
polynomial in the coordinate where $F$ is linearizable. As already seen,
this coordinate is smooth/analytic depending on whether $F$ is smooth/analytic. The proposition then follows in Case~1.

Now we need to consider the case where $g_i (0) =0$ for every~$i$. Concerning the previous discussion,
the only additional difficulty this may pose lies in the fact that a (minimal) positive integer $K_i$ fulfilling
Estimate~(\ref{controllingpolynomial1}) may fail to exist. Indeed, it may even happen that the sequence
of maps $\{ F^{-k} \circ g_i \circ F^k \}_k$, indexed by $k$ and having~$i$ fixed, converges to the identity.
Here two different cases have to be considered.

\noindent {\it Case 2}. Suppose that $K_i$ as before do not exist (for all but finitely many~$i$). In
particular $g_i (0) =0$ for $i$ large enough.

Modulo discarding finitely many values of~$i$, we assume that $K_i$ do not exist for every~$i$. Suppose
first the existence of some~$i_0$ for which the sequence of maps $H_k= F^{-k} \circ g_i \circ F^k$ (indexed by~$k$)
converges to the identity. This case happens for example if there is $i_0$ such that $g_i$ is tangent to
the identity to an order $N$ satisfying $\vert \lambda_n \vert^N < \vert \lambda_1 \vert$.
Nonetheless, for $k$ large,
the value of $\Vert F^{-k} \circ g_{i_0} \circ F^k - {\rm id} \Vert_{r,U}$ is mostly concentrated over a
polynomial part of the Taylor series of $F^{-k} \circ g_{i_0} \circ F^k$. This polynomial part has degree
fixed (i.e. uniformly bounded on $k$) since it is determined solely by the Taylor series of $g_{i_0}$. Thus
Lemma~\ref{vectorfield4} can be employed to ensure that the sequence $H_k$ satisfies the conditions of
Proposition~\ref{vectorfield1} and hence there is a non-identically zero vector field in the $C^{r-1}$-closure
of $G$. To check that $X$ can again be chosen smooth/analytic according to the regularity of $F$, proceed as before
to conclude that each $H_k$ can be represented under the form $\overline{H}_k + R_k$ where the following holds:
\begin{itemize}
\item $\overline{H}_k$ is a polynomial of uniformly bounded degree.
\item The sequence of quotients defined by $\Vert R_i \Vert_{r,U} / \Vert \overline{H}_i \Vert_{0,U}$ converge to zero.
\end{itemize}
In view of this, the previous argument allows to conclude again that $X$ is polynomial in the coordinates where
$F$ is linearizable. The proposition follows in the present case.

Suppose now that for no value $i_0$ of $i$ the sequence (indexed by $k$) of maps
$F^{-k} \circ g_i \circ F^k$ converges to the identity. In particular, there is a uniform $N$ such that
no $g_i$ is tangent to the identity to order~$N$. Now, for~$i$ fixed, consider again the sequence
$h_{i,k} = F^{-k} \circ g_{i} \circ F^k$. As $k$ becomes large,
the value of $\Vert F^{-k} \circ g_{i} \circ F^k - {\rm id} \Vert_{r,U}$ is mostly concentrated over the
polynomial part of degree $N$ of the Taylor series of $F^{-k} \circ g_{i} \circ F^k$. Hence we can choose
a convenient sequence of pairs $(i, k(i))$ leading to maps $h_{i,k(i)}$ converging to the identity and such that
the value of $\Vert F^{-k (i)} \circ g_{i} \circ F^{k(i)} - {\rm id} \Vert_{r,U}$ becomes more and more concentrated
over the degree-$N$ Taylor polynomial of $F^{-k (i)} \circ g_{i} \circ F^{k(i)} - {\rm id}$. Now
the combination of Lemma~\ref{vectorfield4} and Proposition~\ref{vectorfield1} yields again the the existence of
$X$ in the $C^{r-1}$-closure of $G$. The argument used in the preceding cases can similarly be repeated to
show that $X$ is smooth/analytic depending only on the regularity of $F$. The proposition is proved.
\end{proof}

Let us close this section with the proof of Proposition~\ref{vectorfield3star}.

\begin{proof}[Proof of Proposition~\ref{vectorfield3star}]
The argument is analogous though much simpler than the one employed in the proof of Proposition~\ref{vectorfield3}.
In analytic coordinates, $F$ is again represented by the homothety $(x_1, \ldots ,x_n) \longmapsto
(\lambda_1 x_1 , \ldots ,\lambda_n x_n)$. Next, we consider maps $h_{i,k}$ of
the form $h_{i,k} = F^{-k} \circ g_i \circ F^{k}$. Besides, owing to ${\rm HC}$-assumption, each $g_i$
is represented by its Taylor series (based at the origin) on a fixed poly-disc about the origin. Since the
effect of the conjugation by $F^k$ decreases the terms of higher order of the Taylor series of $g_i$ faster
than those of lower orders, the previous scheme allows us to choose a sequence $k(i)$ such that
the corresponding maps $h_i = F^{-k (i)} \circ g_i \circ F^{k (i)}$ satisfy the conditions of
Proposition~\ref{vectorfield1star}. Thus there is a non-identically zero analytic vector field $X$ contained in
the $C^{\omega}$-closure of $G$. The proposition is established.
\end{proof}

\section{Quasi-invariant measures and proof of Theorem~A}

From now on, all groups of diffeomorphisms $G$ are supposed to consist of analytic diffeomorphisms,
i.e. $G \subset \diffM$. Furthermore, {\it every local vector field supposed to be contained in the closure of
$G$ is understood to be non-identically zero}. Unless otherwise stated, these vector fields are also
real analytic. On the other hand, the group
$G$ is supposed to be non-discrete but
only in the $C^{\infty}$-topology. In other words, we shall deal with groups
$G \subset \diffM$ satisfying conditions~(1), (2) and~(3) of Section~2.

This section consists of three paragraphs whose contents are as follows.
In the first paragraph, ideas used in Section~3 will further be developed to yield vector fields in the $C^{\infty}$-closure of
$G$. In the second paragraph, we shall extend to higher dimensions a proposition claiming that
the class of the Lebesgue measure is the only (class of) quasi-invariant measure
for the affine group of the line. A proof for the original proposition can be found in \cite{rebelo-PLMS} and, in
the mentioned paragraph, the corresponding argument
will significantly be generalized to higher dimensions.
Here it may be noted that the main difficulty involved in this generalization
is overcome by the use of a famous result due to Oxtobi and Ulam. Finally, in the last paragraph,
all the previous information will be brought together to prove Theorem~A and to deduce Corollaries~B and~C.

\subsection{Smooth vector fields in the closure of a group}

To begin with, let $F$ denote an analytic Morse-Smale diffeomorphism of $M$. By definition,
the non-wandering set $\Omega \, (F)$ of $F$ must be finite
and thus it must coincide with the set of periodic points ${\rm Per}\, (F)$ of $F$. Moreover, every periodic point $p$ of $F$
is hyperbolic with stable and unstable manifolds denoted respectively by $W^s (p)$ and $W^u (p)$. It is still
part of the definition of a Morse-Smale diffeomorphism that, for every pair of periodic points $p, \, q$
the stable manifold $W^s (p)$ of $p$ and the unstable manifold $W^u (q)$ of $q$ are in general position.
Modulo passing to a finite power of
$F$, the diffeomorphism $F$ possesses both {\it contracting and expanding}\, fixed points.
Note that an expanding fixed point for $F$ is contracting for $F^{-1}$ and
conversely. Let $p_1, \ldots ,p_l$ (resp. $q_1, \ldots ,q_s$) be the contracting fixed points of $F$
(resp. $F^{-1}$). Denote by ${\rm Bas}_{F} (p_i)$ (resp. ${\rm Bas}_{F^{-1}} (q_j)$) the {\it basin
of attraction}\, of $p_i$ (resp. $q_j$), namely the set of points $x \in M$ such that
$\lim_{k \rightarrow \infty} F^k (x) =p_i$ (resp. $\lim_{k \rightarrow \infty} F^{-k} (x) =q_j$).
Now we have the following:

\begin{lema}
\label{vectorfield5}
Suppose that $G \subset \diffM$ satisfies conditions~(1), (2) and~(3) of Section~2. Then the $G$-orbit of
every point $p$ accumulates at one of the points $p_1, \ldots ,p_l, q_1, \ldots ,q_s$.
\end{lema}

\begin{proof}
Consider the union
$$
W = \bigcup_{i=1}^l {\rm Bas}_{F} (p_i) \; \cup \; \bigcup_{j=1}^s  {\rm Bas}_{F^{-1}} (q_j) \, .
$$
The set $W \subset M$ is clearly open. To prove the claim, it suffices to check that the $G$-orbit of every
point in $M$ intersects $W$. For this, note that the
complement $M \setminus W$ of $W$ in $M$ consists of those points $x \in M$
for which there are fixed points $P$ and $Q$ of $F$ satisfying the following conditions.
\begin{itemize}
  \item $P$ and $Q$ are (necessarily hyperbolic) fixed points of $F$ having saddle type (i.e. the corresponding
  differential have eigenvalues of absolute value less than~$1$ and eigenvalues of absolute value greater than~$1$).
  \item The point $x$ belongs to the intersection of $W^u (P) \cap W^s (Q)$.
\end{itemize}
Since $F$ is real analytic, the stable (resp. unstable) manifolds of all its periodic points are analytic as well.
Thus these invariant manifolds are analytic sub-manifolds of $M$. They are, besides, properly embedded
since the $\omega$-limit (resp. $\alpha$-limit) of every point in $M$ must be contained in ${\rm Per}\, (F)$.
Thus, for every pair of periodic points $P, Q$, the transverse intersection $W^u (P) \cap W^s (Q)$
is a proper analytic subset of $M$. In other words, the complement $M \setminus W$ of $W$ in $U$ is a proper
analytic subset of $M$.

To complete the proof of the lemma, we proceed as follows. Consider an irreducible component $A_1$ of $M
\setminus W$. According to condition~(3) in Section~2, there exists $g \in G$
such that $g (A_1)$ is not contained in $M \setminus W$. We then consider the intersection
$g (A_1) \cap (M \setminus W)$ and, if not empty, it yields analytic sets with dimension strictly less than
the dimension of $A_1$. Since the number of irreducible components of $M \setminus W$ is finite, we can repeatedly
use condition~(3) until be able to ensure that the corresponding intersections are empty. The
proof of the lemma is over.
\end{proof}

Since the group $G$ is non-discrete in
the $C^{\infty}$-topology, it would be convenient to have vector fields in the $C^{\infty}$-closure of
$G$ whereas Proposition~\ref{vectorfield3} produces analytic vector fields that are only in the $C^r$-closure of $G$.
Though $r$ can a priori be fixed, the vector field in question depends on this fixed $r$ so that we cannot
immediately derive the existence of a non-trivial $C^{\infty}$-closure for $G$ as desired. Note that working
with the $C^{\infty}$-closure of $G$ is needed to coherently define a (pseudo) Lie algebra. Indeed,
the commutator of two vector fields in the $C^r$-closure of $G$ is, in principle, only in the $C^{r-1}$-closure
of $G$. This gap in the statement of Proposition~\ref{vectorfield3} is however bridged by the theorem below.

\begin{teo}
\label{goodvectorfields}
Assume that $G \subset \diffM$ satisfies conditions~(1), (2) and~(3) of Section~2. Then there exists
a finite open covering $\{ U_k\}_{k=1,\ldots ,l}$ of $M$ such that every $U_k$ is endowed with a nowhere
zero analytic vector field $X_k$ contained in the $C^{\infty}$-closure of $G$.
\end{teo}

\begin{obs}
{\rm If $G$ satisfies conditions~(1), (2) and~(3) of Section~2 but it is constituted by smooth
diffeomorphisms then the statement of Theorem~\ref{goodvectorfields} still holds except that the resulting
vector fields $X_k$ are only smooth. The proof of this differentiable version of the statement can immediately
be obtained from the argument given below modulo replacing the analytic statements from Section~3 by their
analogues in the smooth case.}
\end{obs}

Owing to the compactness of $M$, in order to prove Theorem~\ref{goodvectorfields} it suffices to show that
every point in $M$ belongs to the domain of definition of a nowhere zero analytic vector field in the
$C^{\infty}$-closure of $G$. In turn, Lemma~\ref{vectorfield5} reduces the latter statement to showing that
every contracting fixed point $p$ of $F$ lies in the domain of definition of an analytic vector field $X$ in the
$C^{\infty}$-closure of $G$ and satisfying $X (p) \neq 0$. Let then $p$ as above be fixed.
First we have:

\begin{lema}
\label{morevectorfield1}
Fixed $r \in \N \cup \{ \infty\}$, there is an analytic vector field $Y_r$ defined about $p$ and verifying $Y_r(p) \neq 0$
that is contained in the $C^r$-closure of $G$.
\end{lema}

\begin{proof}
Consider the set
$\mathcal{S} \subset M$ consisting of those points $q \in M$ verifying the following condition: every analytic vector field
$X$ defined about $q$ and contained in the $C^r$-closure of $G$ is such that $X (q) =0$.
Since the combination of Proposition~\ref{vectorfield3} and Lemma~\ref{vectorfield5} implies that every point
in $M$ belongs to the domain of definition of a
(non-identically zero as always) analytic vector field contained in the $C^r$-closure of $G$, it follows that
$\mathcal{S}$ is a proper analytic set of $M$. Besides, $\mathcal{S}$ is clearly invariant by $G$. Therefore,
in view of condition~(3), the set $\mathcal{S}$ must be empty and the lemma follows.
\end{proof}

As mentioned, Theorem~\ref{goodvectorfields} amounts to finding an analytic vector field $X$ in the
$C^{\infty}$-closure of $G$ such that $X (p) \neq 0$.
The argument provided in Lemma~\ref{morevectorfield1} (in the case $r =\infty$), in turn, ensures
that it suffices to construct a non-identically zero analytic vector field $X$ defined on a neighborhood of $p$
and contained in the $C^{\infty}$-closure of $G$.

The idea to prove Theorem~\ref{goodvectorfields} is to construct a single vector field
$X$ out of the vector fields $Y_r$ for every $r \in \N$,
that will be in the $C^r$-closure of $G$ for every~$r$. However, before providing details on this construction,
some additional terminology will be needed.

Consider a vector-monomial $Z$, i.e. a vector field of the form $Z= c x_1^{\alpha_1} \ldots x_n^{\alpha_n} \partial /\partial x_l$,
with $\alpha_i \in \N$ for every $i \in \{ 1, \ldots ,n\}$ and $c \in \C$. Recall that, in the coordinates
$(x_1, \ldots ,x_n)$, we have $F(x) =F (x_1, \ldots ,x_n) = (\lambda_1 x_1 , \ldots ,\lambda_n x_n)$ so that
$F^{\ast} Z$ is a constant multiple of $Z$. Let $c_F$ be the value of this constant so that $F^{\ast} Z = c_F .Z$.
The {\it $F$-multiplier}\, $m(F)$ of $Z$ is nothing but the absolute value $\vert c_F \vert$ of $c_F$.
A polynomial vector field is said to be $F$-homogeneous
if all its vector monomials have the same $F$-multiplier. Besides, if $X = \sum_{l=1}^n \sum_{\alpha_1,
\ldots ,\alpha_n \in \N} c_{\alpha_1, \ldots ,\alpha_n}^l x_1^{\alpha_1} \ldots x_n^{\alpha_n} \partial /\partial x_l$
is an analytic vector field and $c x_1^{\alpha_1} \ldots x_n^{\alpha_n} \partial /\partial x_l$ is a vector-monomial
of its Taylor series, then the $F$-homogeneous component of $c x_1^{\alpha_1} \ldots x_n^{\alpha_n} \partial /\partial x_l$
is the $F$-homogeneous vector field consisting of all vector monomials appearing in the Taylor series of $F$ and
having the same $F$-multiplier as $c x_1^{\alpha_1} \ldots x_n^{\alpha_n} \partial /\partial x_l$. The fact that
$\vert \lambda_1 \vert < \cdots < \vert \lambda_n \vert < 1$ ensures that the $F$-homogeneous component of every
vector-monomial appearing in a Taylor series is necessarily a polynomial vector field.

\begin{obs}
{\rm Consider the space $\mathcal{G}$ of (smooth or analytic) vector fields defined on some open set
about a point $q \in M$ and contained in the $C^{\infty}$-closure of $G$. This space is clearly a
(pseudo) Lie algebra over $\R$ though it may, in principle, be trivial. In any event, being an algebra
over $\R$,
every real multiple $cX$ of a vector field $X \in \mathcal{G}$ also lies in $\mathcal{G}$.
This is however not the case if $c$ is allowed
to be an arbitrary complex number. Since, in our discussion, $F$ is
diagonalized over $\C$, this issue would require us to perform some
additional verification whenever multiples of vector fields are considered in the course
of this section. These verifications however will systematically
be left to the reader since they are straightforward. Alternatively, instead of taking $F$ diagonal over $\C$, the
reader may work with the standard ``block-diagonal'' that avoids
the use of complex numbers at the expenses of making the notations slightly cumbersome.}
\end{obs}

We are now able to prove Theorem~\ref{goodvectorfields}.

\begin{proof}[Proof of Theorem~\ref{goodvectorfields}]
According to Proposition~\ref{vectorfield3} (and Lemma~\ref{morevectorfield1}),
for every $r \in \N$ there exists an analytic vector field $Y_r$ in the
$C^r$-closure of $G$ and verifying $Y_r (p) \neq 0$ (so that $Y_r$ is not identically zero). For $r$ fixed,
let $Z^1_r, Z^2_r, \ldots $ denote the $F$-homogeneous components of the (analytic) vector field $Y_r$ labeled
in decreasing order of the corresponding $F$-multipliers. In other words, if $m^j(F)$ stands for the
$F$-multiplier of $Z_r^j$, then $m^{j_1}(F) > m^{j_2}(F)$ provided that $j_1 < j_2$. Now we have:

\vspace{0.1cm}

\noindent {\it Claim~1}. The sequence of polynomial vector fields $\{ Z_r^1\}_{r \in \N}$ has uniformly
bounded degree.

\noindent {\it Proof of Claim~1}. Note that the set of all $F$-multipliers associated to vector-monomials
is fully constituted by values of the form $\vert \lambda_1^{\alpha_1} \ldots \lambda_n^{\alpha_n} /\lambda_l \vert$,
where $\alpha_i \in \N$ for every $i \in \{ 1, \ldots ,n\}$ and $l \in \{ 1, \ldots ,n\}$.
Since $Y_r (0) \neq 0$, $Y_r$ has a homogeneous component of multiplier
greater than or equal to~$1/\vert \lambda_1 \vert$ (recall that $\vert \lambda_1 \vert < \cdots < \vert \lambda_n \vert < 1$).
Thus there may exist only finitely many values of the form $\vert \lambda_1^{\alpha_1} \ldots \lambda_n^{\alpha_n} /\lambda_l
\vert$ that are greater than~$1/\vert \lambda_1 \vert$ and these values depend only on
$\lambda_1, \ldots ,\lambda_n$ (in particular they do not depend on~$r$). The claim follows.\qed

Let $d$ denote the degree of the first $F$-homogeneous component $Z^1_r$ of $Y_r$.
We shall prove:

\vspace{0.1cm}

\noindent {\it Claim~2}. There exists a (non-identically zero) polynomial vector field of degree
at most~$d$ which is contained in the $C^r$-closure of $G$.

\noindent {\it Proof of Claim~2}. Let $r$ and $Y_r$ be fixed. Denote by $m^j (F)$ the $F$-multiplier of
$Z_r^j$ for $j \in \N^{\ast}$.
For $j$ fixed, the $j^{\rm th}$-$F$-homogeneous component $Z^j_r$ of $Y_r$ can be split into a sum $\sum_{l=1}^{m(j)}
Z_r^{j,l}$ of vector monomials. Now consider the sequence $\{ Y_r^k\}_{k \in \N}$ (indexed by $k$) of vector fields
defined by $Y_r^k = (F^k)^{\ast} Y_r$. With the previous notations, we have
$$
Y_r^k = (F^k)^{\ast} Y_r = [m^1 (F)]^k  \, ( \sum_{l=1}^{m(1)} c^{1,l} Z_r^{1,l} )
+ \sum_{j \geq 2} \, [m^j (F)]^k ( \sum_{l=1}^{m(j)} c^{j,l} Z_r^{j,l} )
$$
where all the coefficients $c^{1,l}$ and $c^{j,l}$ are complex numbers lying in the unit circle.
Clearly all the vector fields $Y_r^k$ belong to the $C^r$-closure of $G$. Since constant multiples of vector fields
in the closure of $G$ also belong to the closure of $G$, it follows that the new vector fields $\overline{Y}_r^k$
given by $\overline{Y}_r^k = (F^k)^{\ast} Y_r /(m^1 (F))^k$ belong to the $C^r$-closure of $G$ as well. These latter
vector fields are however given by
$$
\overline{Y}_r^k = \sum_{l=1}^{m(1)} c^{1,l} Z_r^{1,l} + \sum_{j\geq 2} \left[
\frac{m^j (F)}{m^1 (F)} \right]^k ( \sum_{l=1}^{m(j)} c^{j,l} Z_r^{j,l}) \, .
$$
In particular, by sending $k \rightarrow \infty$, it follows that $\{ \overline{Y}_r^k \}$ forms a compact family
whose non-trivial accumulation points in the $C^{\infty}$-topology are all contained in the space of polynomial
vector fields with degree at most~$d$. These accumulation points are not identically zero vector fields since all
the vector monomials appearing in $Z^1_r$ will also appear in these limits after being multiplied by some
complex number with absolute value equal to~$1$. Next note that every converging subsequence of
$\{ \overline{Y}_r^k \}$ will converge to its limit (polynomial) vector field not only in the $C^r$ but also
in the $C^{\infty}$-topology since
the latter is controlled by the topology of (the finitely many) coefficients. Because the $C^r$-closure of $G$ is naturally
closed for the $C^r$-topology, and given that $\overline{Y}_r^k$ belongs to the $C^r$-closure of $G$ for every~$k$,
it follows that all the above mentioned limits belong to the $C^r$-closure of $G$ as well.
The claim is proved.\qed

Let us now conclude the existence of a non-identically zero analytic vector field defined about $p \simeq 0 \in \R^n$
and contained in the $C^{\infty}$-closure of $G$. For each $r$, let $\overline{Z}_1^r$ be a polynomial vector field
belonging to the $C^r$-closure of $G$ and obtained from the construction given in the proof of Claim~2.
Consider the sequence (indexed by $r \in \N$) of vector fields
$\overline{Z}_r^1$. It follows from Claim~1 that all the (polynomial) vector fields in this sequence have
degree uniformly bounded by some constant~$N \in \N$. Moreover, modulo multiplying each $\overline{Z}_r^i$ by suitable
constants, we can assume without loss of generality that the maximum of the norm of the corresponding
coefficients is exactly~$1$. Now, since the space of polynomials with bounded degree is locally compact,
we can suppose that the vector fields $\overline{Z}_r^1$ converge (even in the analytic topology) towards a polynomial
vector field $\overline{Z}_{\infty}$ which does not vanish identically since at least one of its coefficient has absolute
value equal to~$1$. It remains only to check that $\overline{Z}_{\infty}$ lies in the $C^{\infty}$-closure of $G$.
This is however easy: note that, fixed $r_0 \in \N$, the vector fields $\overline{Z}_r^1$ belong to the $C^{r_0}$-closure
of $G$ provided that $r \geq r_0$. Since the latter converge in the $C^{\infty}$-topology towards $\overline{Z}_{\infty}$,
it follows that $\overline{Z}_{\infty}$ belongs to the $C^{r_0}$-closure of $G$. Since $r_0$ is arbitrary, we conclude that
$\overline{Z}_{\infty}$ belongs to the $C^{\infty}$-closure of $G$ completing the proof of the theorem.\
\end{proof}

In closing this paragraph, let us provide a useful and fairly general lemma. Note that the
first intrinsic advantage of considering (analytic) vector fields in the $C^{\infty}$-closure of $G$, as opposed
to vector fields in the $C^r$-closure of $G$, lies in the fact that the former clearly constitutes a Lie (pseudo-)
algebra. Moreover, letting $p$ be
a fixed contracting point of $F$ and choosing local coordinates where
$F$ becomes the homothety $(x_1, \ldots ,x_n) \mapsto (\lambda_1 x_1 , \ldots ,\lambda_n x_n)$, we have
the following.

\begin{lema}
\label{grading}
If $X$ is an analytic vector field lying in the $C^{\infty}$-closure of $G$, then every $F$-homogeneous
component of $X$ belongs to the $C^{\infty}$-closure of $G$ as well.
\end{lema}

\begin{proof}
Let $Z^1, Z^2, \ldots$ denote the $F$-homogeneous components of $X$ labeled in terms of decreasing $F$-multipliers.
The central point is to show that the $F$-homogeneous component $Z^1$ belongs to the $C^{\infty}$-closure of $G$. Once
this has been established, the lemma follows by inductively considering the vector fields
$X-Z^1$, $X-Z^1-Z^2$ and so on.

To check that $Z^1$ lies in the $C^{\infty}$-closure of $G$, consider $Z^1$ written as a finite sum of vector-monomials
$$
Z^1 = \sum_{l=1}^N a_{\alpha_1(l), \ldots , \alpha_n(l)} x_1^{\alpha_1 (l)} \ldots x_n^{\alpha_1 (l)} \partial /
\partial x_{\beta(l)} \, .
$$
For $l$ fixed, the value of the constant $c_F$ for which the monomial in question is multiplied when pulled-back
by $F$ is going to be denoted by $c_l$. In particular, for every $l$ we have $c_l =m^1(F) e^{\theta_l \sqrt{-1}}$
where $m^1 (F)$ is the $F$-multiplier of $Z^1$. As previously seen, the sequence $Y_k =(F^k)^{\ast} X /m^1 (F)$
forms a compact family whose non-trivial limit points are contained in the space of polynomial vector fields
with degree bounded by the degree of $Z^1$. Next note that, if all the $\theta_l$ are rational, then
there exists a subsequence of $\{ Y_k\}$ converging towards a constant multiple of $Z^1$ and the statement follows
in this case.

The general case can be treated by induction on the number of vector monomials appearing in $Z^1$. To simplify
the discussion, suppose that this number is two, i.e. $Z^1 = Z^1_1 + Z^1_2$ where $Z^1_1 ,\, Z^1_2$ are
vector-monomials respectively satisfying $F^{\ast} Z^1_1 = m^1 (F) e^{\theta_1 \sqrt{-1}}$ and
$F^{\ast} Z^1_2 = m^1 (F) e^{\theta_2 \sqrt{-1}}$. It is then clear that the sequence of vector fields
$Y_k = (F^k)^{\ast} X /[m^1 (F)]^k e^{k\theta_1 \sqrt{-1}}$ has its (proper) accumulation points contained in
the set of vector fields having the form
$$
Z_1^1 + e^{k(\theta_2-\theta_1) \sqrt{-1}} Z_2^1 \, .
$$
By taking the difference between two conveniently chosen limits as above, we conclude that $Z^1_2$ lies
in the $C^{\infty}$-closure of $G$. In turn, this implies that the same holds for $Z^1_1$ and hence
$Z^1 = Z^1_1 + Z^1_2$ must belong to the $C^{\infty}$-closure of $G$ as well. Now a standard induction argument
completes the proof of the lemma.
\end{proof}

\subsection{Extensions of the translation group of $\R^n$ and quasi-invariant measures}

Let ${\bf T}$ denote the group of translations of $\R^n$ and consider another $C^{\infty}$ (or $C^{\omega}$)
diffeomorphism $f: \R^n \rightarrow \R^n$ having a contracting fixed point at the origin. The {\it free product}\,
between $\Z$ and ${\bf T}$ is going to be denoted by $\Z{\bf T}$. The choice of $f$ enables us to define
a representation $\rho_0 : \Z {\bf T} \rightarrow {\rm Diff}\, (\R^n)$ by setting $\rho_0 (1\ast 0) = f$
and $\rho_0 (0\ast B)$ equal to the translation $x \mapsto x+B$. The image of $\Z {\bf T}$ by $\rho_0$ is going to
be denoted by $\Gamma_0$. Naturally we have $f \in \Gamma_0$ and ${\bf T} \subset \Gamma_0$.

Note that the existence of elements with contracting fixed points in $\Gamma_0$ easily implies that
this group does not possess any invariant sigma-finite measure on $\R^n$.
Thus it is natural to look for $\Gamma_0$-invariant {\it classes of measures}\, on $\R^n$. These are fully
characterized by the proposition below.

\begin{prop}
\label{uniquenesslebesgue}
Every probability measure on $\R^n$ that is quasi-invariant by the action $\Gamma_0$ is absolutely continuous with
respect to the Lebesgue measure.
\end{prop}

The usual Lebesgue measure $\Leb$ on $\R^n$ is clearly quasi-invariant under the action of $\Gamma_0$. Since we are
dealing only with quasi-invariant measures, it is useful to consider also another Lebesgue
measure $\leb$ on $\R^n$ arising from normalizing $\Leb$ by some Gaussian kernel with radial symmetry so as
to have $\leb (\R^n) =1$. Let now $\mu$ denote
an arbitrary probability measure on $\R^n$ that also happens to be quasi-invariant under the action $\Gamma_0$. Naturally $\mu$
cannot give mass to points since $\Gamma_0$ is uncountable. Furthermore, being quasi-invariant by $\Gamma_0$, the support
of $\mu$ coincides with all of $\R^n$. At this point, a deep theorem due to Oxtobi and Ulam \cite{OxU}
guarantees the existence of a homeomorphism
$H: \R^n \rightarrow \R^n$ such that $\leb = H^{\ast} \mu$, i.e. for every Borel set $\calb \subseteq \R$ we have
\begin{equation}
\leb \, (H^{-1} (\calb)) = \mu  ( \calb) \, . \label{muH}
\end{equation}

By using the homeomorphism $H$ a new representation $\rho_1$ from $\Z {\bf T}$ to ${\rm Homeo}\, (\R^n)$
can be constructed by letting $\rho_1 = H^{-1} \circ \rho_0 \circ H$. The image of $\Z {\bf T}$ by $\rho_1$
will be denoted by $\Gamma_1 \subset {\rm Homeo}\, (\R^n)$. Clearly $\Gamma_0, \, \Gamma_1$ are isomorphic
as abstract groups. Next, we claim that every element $\gamma_1 \in \Gamma_1$ is an absolutely continuous
homeomorphism of $\R^n$. To check the claim, denote by $\gamma_0$ the element of $\Gamma_0$ satisfying
$\gamma_1 = H^{-1} \circ \gamma_0 \circ H$. Suppose then that $\calb$ is a Borel set
such that $\leb \, (\calb) =0$. Since $H$ is invertible, $\calb = H^{-1} (H (\calb))$ and thus $\mu (H (\calb))=0$.
However $\mu$ is quasi-invariant under $\Gamma_0$ so that $\mu [ \gamma_0 \circ H (\calb)] =0$.
In turn, Equation~(\ref{muH}) provides
$\leb \, [(H^{-1} \circ \gamma_0 \circ H (\calb)] =0$ i.e. $\leb \, [ \gamma_1  (\calb)]=0$. Since $\gamma_1$
is an arbitrary element of $\Gamma_1$ the claim follows at once.

Going back to $\Gamma_0$, the group of translations ${\bf T}$ is naturally a subgroup of $\Gamma_0$. As to
$\Gamma_1$, the image under $\rho_1$ of $0 \ast {\bf T} \subset \Z {\bf T}$ yields a commutative and simply
transitive group of (absolutely continuous) homeomorphism of $\R^n$ that is going to be denoted by ${\bf T}^H$.
Note that the action of ${\bf T}^H$ is naturally identified to $H^{-1} \circ {\bf T} \circ H$. Since ${\bf T}$
preserves the Lebesgue measure $\Leb$ on $\R^n$, the group ${\bf T}^H$ preserves the pull-back $H^{\ast} \Leb$
of $\Leb$ by $H$. In particular ${\bf T}^H$ possesses a sigma-finite invariant measure. Now we state:

\begin{lema}
\label{qinvariant1}
There is a sigma-finite measure $\nu$ on $\R^n$ satisfying the following conditions:
\begin{enumerate}
  \item $\nu$ is invariant under ${\bf T}^H$.
  \item $\nu$ is absolutely continuous with respect to $\Leb$ (or, equivalently, to $\leb$).
\end{enumerate}
\end{lema}

\begin{proof}
Note that ${\bf T}^H$ is the image of the continuous homomorphism
$\rho : \R^n \rightarrow {\rm Homeo}\, (\R^n)$ given by $\rho (B) (x) = H^{-1} (H(x) + B)$ since $H$ conjugates
${\bf T}$ and ${\bf T}^H$. Besides, we already
know that the element $\rho (B) \in {\rm Homeo}\, (\R^n)$ is absolutely continuous for every $B \in \R^n$.
Therefore, given $B \in \R^n$, the Radon-Nikodym derivative
of $[\rho (B)]^{\ast} \Leb$ with respect to $\Leb$ at the point
at $x=(x_1, \ldots ,x_n) \in \R^n$ can be considered. The resulting function, denoted by ${\rm Det} \, (\rho (B)) : \R^n \rightarrow \R$,
is clearly nonnegative and belongs to $L_{\rm loc}^1$ (where, as
usual, $L^1, \, L_{\rm loc}^1$ are defined with respect to the Lebesgue measure).

Let us now consider the ``evaluation map'' at $z \in \R^n$ defined by $B \mapsto {\rm Det} \, (\rho (B)) (z)$
where $z \in \R^n$ is fixed. Now we have:

\noindent {\it Claim}: For almost all $z \in \R$ the evaluation map $B \mapsto {\rm Det} \, (\rho (B)) (z)$ belongs to
$L_{\rm loc}^1$.

\noindent {\it Proof of the Claim}. Consider two points $p = (p_1, \ldots ,p_n)$ and $q=(q_1, \ldots ,q_n)$ of $\R^n$.
Let $B =(b_1, \ldots ,b_n)$ be fixed. The statement would follows from Fubini's theorem, if we can ensure
that the ``double integral''
$$
\iint\limits_{(z, B) \in \R^n \times \R^n} {\rm Det} \, (\rho (B)) (z) dB dz
$$
is well-defined. Nonetheless, since ${\rm Det} \, (\rho (B))$ is nonnegative,
the classical Tonelli's theorem (cf. \cite{bartle}, page~118) reduces this proof to showing that the integral
$$
\int_{p_1}^{q_1} \cdots \int_{p_n}^{q_n}  {\rm Det} \, (\rho (B)) (z) dz
$$
exists for every $B \in \R^n$. By construction, the latter integral is nothing but the Lebesgue
measure $\Leb$ of the solid $H ([p_1,q_1] \times \cdots \times [p_n,q_n])$ and hence it is well-defined.
The claim is proved.\qed

Fix $z \in \R^n$ such that the evaluation map $B \mapsto {\rm Det} \, (\rho (B)) (x)$
belongs to $L_{\rm loc}^1$. Next let $\digamma : \R^n \rightarrow \R^n$ be defined by $\digamma (B) = c (B). B$,
where $B \in \R^n$ and where $c (B) \in \R$ is given by
$$
c (B) =  \int_{0}^B {\rm Det} \, (\rho (x_1, \ldots ,x_n)) (z) \, dx_1 \ldots dx_n \, .
$$
Here ``$0$'' stands for the origin of $\R^n$ while the integration path is the radial line from the origin
to $B \in \R^n$. The map $\digamma$ is clearly continuous and its restriction to real lines through the origin
is monotone in the natural sense. It then follows that $\digamma$ is one-to-one and hence it is a homeomorphism
onto its image contained in $\R^n$. Furthermore the homeomorphism $\digamma$ is, indeed, absolutely continuous
since the map $B \mapsto {\rm Det} \, (\rho (B)) (z)$ belongs to $L_{\rm loc}^1$. Therefore the measure
$\nu = \digamma^{\ast} \Leb$ is absolutely continuous with respect to $\Leb$. Finally, $\nu$ is invariant by
$D^1 [\rho_1 ((\R^{\ast})^n \times \R^n)]$ since $\digamma$ conjugates $D^1 [\rho_1 ((\R^{\ast})^n \times \R^n)]$
to the group of translations of $\R^n$ and the latter preserves $\Leb$. In other words, the measure $\nu$
fulfils the conditions in the statement.
\end{proof}

We can now complete the proof of Proposition~\ref{uniquenesslebesgue}.

\begin{proof}[Proof of Proposition~\ref{uniquenesslebesgue}]
Consider the measure $\nu$ whose existence is ensured by Lemma~\ref{qinvariant1}.
The push-forward $H_{\ast} \nu$ of $\nu$ by $H$ is a sigma-finite measure on $\R^n$ that
happens to be invariant by ${\bf T}$, i.e. by all translations
of $\R^n$. Due to the standard uniqueness property of the Haar measure, it follows that
$H_{\ast} \nu$ must be a constant multiple of $\Leb$. Thus $H^{\ast} \Leb$ is again an absolute continuous measure
since it coincides with a
constant multiple of $\nu$. This means that $H$ is an absolutely continuous {\it homeomorphism}. Since
$\mu = H_{\ast} \leb$, with $H$ absolutely continuous, we conclude that $\mu$ must be absolutely continuous
with respect to $\Leb$. The proposition is proved.
\end{proof}

Here is an immediate corollary of Proposition~\ref{uniquenesslebesgue} that is worth stating in view of
the toy-theorem mentioned in the Introduction.

\begin{coro}
\label{forpsl2c}
Assume that $G$ is a dense subgroup of ${\rm PSL}\, (2, \C)$ projectively acting on $S^2$. Suppose that
$\mu$ is a probability measure on $S^2$ that happens to be quasi-invariant by closure of $G$ in ${\rm PSL}\, (2, \C)$.
Then $\mu$ is absolutely continuous.
\end{coro}

\begin{proof}
By assumption the closure of $G$ is all of ${\rm PSL}\, (2, \C)$. In particular it contains a copy of the affine
group ${\rm Aff}\, (\C)$ of the complex line. According to Proposition~\ref{uniquenesslebesgue},
a measure quasi-invariant by ${\rm Aff}\, (\C)$ must be absolutely continuous. The statement follows at once.
\end{proof}

For the sake of completeness, we also mention the proof of Linear Theorem stated in the Introduction.

\vspace{0.1cm}

\begin{proof}[Proof of Linear Theorem]
It is an immediate combination of Proposition~\ref{forpsl2c} with Theorem~D to be proved in Section~5.
\end{proof}

\subsection{Proof of Theorem~A}

In what concerns the existence of vector fields in the closure of groups, the main result obtained in
this paper is Theorem~\ref{goodvectorfields}. Once one vector field in the $C^{\infty}$-closure of $G$
does exist, further generic conditions imposed on the group $G$ allow us to construct additional
vector fields possessing similar properties. The choice of the desired ``generic'' conditions is rather flexible
and does not constitute the main issue in most applications. Here we shall proceed from a slightly more conceptual
point of view directly inspired by \cite{belliart-2}.

As before, we fix a local coordinate about a contracting fixed point of $F$ so that $F$ becomes the homothety
$(x_1, \ldots ,x_n) \mapsto (\lambda_1 x_1 , \ldots ,\lambda_n x_n)$. Consider the (pseudo-) Lie algebra
$\mathcal{G}$ of
analytic vector fields about $p$ that are in the $C^{\infty}$-closure of $G$. As it will be seen, this
Lie algebra can be {\it filtered}\, by means of $F$-homogeneous components. Concerning filtered and/or graded
Lie algebras, the reader is referred to \cite{demazure} for the relevant definitions and results used below.
To begin with, denote by $m^1 (F)> m^2 (F)>, \ldots$ the collection of absolute values of all complex numbers
having the form $\lambda_1^{\alpha_1} \ldots \lambda_n^{\alpha_n}/\lambda_l$, with $\alpha_i \in \N$ and $l \in
\{ 1, \ldots ,n\}$. Fixed $j \in \N^{\ast}$, denote by $\mathcal{V}^j$ the vector space spanned by all vector
monomials whose $F$-multipliers are equal to~$m^j (F)$. As already pointed out, $\mathcal{V}^j$ has finite dimension
as vector space for every~$j \in \N$. Next set $\mathcal{G}^j = \mathcal{G} \cap \mathcal{V}^j$. Owing to
Lemma~\ref{grading}, it follows that $\mathcal{G} = \bigoplus_{j \in \N^{\ast}} \mathcal{G}^j$ where each
$\mathcal{G}^j$ is a finite dimensional vector space. Next we claim that $[ X^{j_1},X^{j_2} ]$ lies in
$\mathcal{G}^{j_1+j_2}$ provided that $X^{j_1} \in \mathcal{G}^{j_1}$ and $X^{j_2} \in \mathcal{G}^{j_2}$.
This is an immediate consequence of the general formula $F^{\ast} [X,Y] = [F^{\ast} X, F^{\ast} Y]$ which,
in our case, yields $F^{\ast} [ X^{j_1},X^{j_2} ] = m^{j_1} (F) m^{j_2} (F) [ X^{j_1},X^{j_2} ]$, i.e.
$m^{j_1} (F) m^{j_2} (F)$ is the $F$-multiplier associated to $[ X^ {j_1},X^{j_2} ]$. An elementary counting
of possibilities then shows that the multiplier in question cannot appear before the position $j_1+j_2$
in the list $m^1 (F), m^2(F), \ldots$ what, in turn, establishes the claim.

Summarizing what precedes, we have proved the following.

\begin{prop}
\label{gradedalgebra}
The Lie (pseudo) algebra of analytic vector fields defined about $p$ and contained in the
$C^{\infty}$-closure of $G$ is graded.\qed
\end{prop}

So far, we have always worked with a subgroup $G \subset \diffM$ satisfying conditions~(1), (2), (3)
of Section~2. Let us now assume that condition~(4) is also satisfied. Then we have.

\begin{lema}
\label{transitivity}
Suppose that $G \subset \diffM$ satisfies conditions~(1) - (4) of Section~2. Then, for every $q \in M$,
there are vector fields $X_1, \ldots ,X_n$ in the $C^{\infty}$-closure of $G$ such that
$X_1 (q), \ldots ,X_n (q)$ form a basis for $T_qM$ (in particular these vector fields are defined about $q$).
\end{lema}

\begin{proof}
In view of Lemma~\ref{vectorfield5}, it suffices to prove the lemma for a contracting fixed point
$p$ of $F$. Therefore let us suppose for a contradiction that the statement is false. Let $X_1, \ldots ,X_k$
be a {\it maximal}\, set of analytic vector fields in the $C^{\infty}$-closure of $G$ about $p$ that happen
to be linearly independent at some point near $p$. Proposition~\ref{goodvectorfields} ensures that $k \geq 1$.
Furthermore the (possibly singular) plane distribution $\mathcal{D}$ associated to $X_1, \ldots ,X_k$ is
integrable since vector fields in the closure of $\mathcal{G}$ form a Lie algebra and $k$ is maximal.

\noindent {\it Claim}. We have $k=n$.

\noindent {\it Proof of the Claim}. Suppose we had $k <n$ (strictly). The maximality of $k$ implies that
$\mathcal{D}$ must be fully invariant by $G$ and hence it induces a singular foliation $\mathcal{F}$ on
the open set $U \subset M$ corresponding to the $G$-orbit of some small neighborhood of $p$. Next note that
$\mathcal{F}$ is non-singular otherwise its singular set would be a proper analytic set of $M$ invariant
by $G$ what is impossible in view of condition~(3). Nonetheless the existence of a non-singular
foliation as above contradicts condition~(4). We then conclude that $k$ cannot be smaller than~$n$ proving
the statement.\qed

Now consider vector fields $X_1, \ldots ,X_n$ that are contained in the $C^{\infty}$-closure of $G$ on
a neighborhood of $p$ and that are linearly independent at generic points in this neighborhood. If the
vector fields $X_1, \ldots ,X_n$ can be chosen linearly independent
at $p$, then the lemma is proved. Suppose that this is not the case. Therefore the non-empty set of
points $q$ in the $G$-invariant set $U \subset M$
at which the rank of the subspace of $T_qM$ spanned by vector fields in the Lie algebra in question is less
than~$n$ constitutes a proper analytic subset of $M$ invariant by $G$. Again this
contradicts condition~(3). The lemma is proved.
\end{proof}

\begin{lema}
\label{localliegroup}
With the preceding notations, suppose that $G \subset \diffM$ satisfies conditions~(1) - (5) of Section~2.
Then the Lie algebra $\mathcal{G}$ consisting of analytic vector field in the $C^{\infty}$-closure
of $G$ about a contracting fixed point $p$ contains all analytic vector fields defined about $p$.
\end{lema}

\begin{proof}
Note that $\mathcal{G}$ is a graded Lie algebra in view of Lemma~\ref{grading}. It is also transitive
thanks to Lemma~\ref{transitivity}. If the dimension of $\mathcal{G}$ can be shown to be infinite, then it will
also follows that $\mathcal{G}$ is irreducible. These Lie algebras were classified by
E. Cartan (note that vector fields in $\mathcal{G}$ are analytic). From his classification it results that
$\mathcal{G}$ has to consist of all analytic vector fields about $p$ since the remaining possibilities
have finite dimension. The lemma will then be proved.

Let us then suppose for a contradiction that the dimension of $\mathcal{G}$ is finite.
We also fix a small neighborhood $U$ of $p$ and denote by $G_U$ the pseudogroup of local diffeomorphism
defined on $U$ by restrictions of elements in $G$. Consider now the abstract finite dimensional Lie algebra
$\mathcal{L}$ isomorphic to $\mathcal{G}$. Every element $g$ in $G_U$ acts linearly on $\mathcal{G}$ by pull-backs
and, therefore, can be identified to an element of the automorphism group ${\rm Aut}\, (\mathcal{L})$ of
$\mathcal{L}$. Now, note that ${\rm Aut}\, (\mathcal{L})$ is itself a Lie group of finite dimension and the
pseudogroup $\Gamma_U$ associated to the image of $G_U$ in ${\rm Aut}\, (\mathcal{L})$
can be considered. Similarly we denote by $\Gamma$ the group arising from the global realization of $\Gamma_U$.
However, by construction, there is a local action of $\Gamma_U$ on $U$.
Furthermore, this local action can naturally be extended to the closure $\overline{\Gamma}_U$ of $\Gamma_U$ in
${\rm Aut}\, (\mathcal{L})$. However $\overline{\Gamma}_U$ is contained in the closure $\overline{\Gamma}$
of $\Gamma$ and $\overline{\Gamma}$ is, in its own right, a finite dimensional Lie group since it is a closed
subgroup of ${\rm Aut}\, (\mathcal{L})$. Assembling all this information, we conclude that $G$ is locally conjugate
about $p$ to a finite dimensional Lie group. This however contradicts condition~(5) and establishes the
lemma.
\end{proof}

\begin{proof}[Proof of Theorem~A] Suppose that $\mu$ is a probability measure on $M$ quasi-invariant
under the $C^{\infty}$-close of $G$. Standard arguments based on Lemma~\ref{transitivity} makes it clear that
the orbit under $G$ of
every point $q \in M$ is dense in $M$. Similarly, $G$ is ergodic with respect to the Lebesgue
measure on $M$. So, from now on, we assume for a contradiction that $\mu$ is singular with respect to the
Lebesgue measure on $M$. Since $G$ is ergodic, there is a Borel set $\calb \subset M$ having null Lebesgue
measure and satisfying $\mu (\calb) =1$. Besides, if $U$ is a non-empty open subset of $M$, then
$\mu (\calb \cap U) >0$ since $M$ is covered by the $G$-orbit of $U$.

Consider now a contracting fixed point $p$ for $F$. Let $U$ be a small neighborhood of $p$ identified to all of
$\R^n$ through some local $C^{\infty}$-coordinate $\varphi$. Modulo normalizing the restriction of $\mu$
to $U$, we can assume that $U$ is equipped with a probability measure $\nu$ which is quasi-invariant
by the pseudogroup $\overline{G}_U$ induced by the restriction to $U$ of all elements in the $C^{\infty}$-closure
of $G$. By setting $\calb' = \calb \cap U$, we have $\nu (\calb')=1$ whereas the Lebesgue measure of $\calb'$
equals~{\it zero}.

By means of the coordinate $\varphi$, $\overline{G}_U$ becomes identified to a pseudogroup of $C^{\infty}$-diffeomorphisms
de $\R^n$ containing the contraction corresponding to $F$ as well
as the whole translation group ${\bf T}$, cf. Lemma~\ref{localliegroup}.
It follows from Proposition~\ref{uniquenesslebesgue} that $\nu$ must be absolute continuous since it is quasi-invariant
by the latter pseudogroup. The resulting contradiction establishes the theorem.
\end{proof}

Now we have:

\begin{proof}[Proof of Corollary~B]
Let $\mu$ be as in the statement and consider a
Borel set $\calb \subset S^1$ such that $\mu (\calb) >0$. Since $\mu$ is supposed to be a
$d$-quasi-volume for $G$, there is a uniform constant $C >0$ such that
\begin{equation}
\mu (g (\calb)) \geq C \int_{\calb} \Vert {\rm Jac}\, [Dg] (x) \Vert^d \, d\mu \label{corollaryb}
\end{equation}
for every $g \in G$.
Next let $\overline{g} : U \subseteq M \rightarrow M$ be a diffeomorphism in the
closure of $G$ and defined on an open set $U$ containing $\calb$. To conclude that $\mu$
is quasi-invariant by elements in the $C^{\infty}$-closure of $G$, we need to show that
$\mu (\overline{g} (\calb)) > 0$.
For this, fix a sequence $\{ g_i \}$
of actual elements in $G$ whose restrictions $g_{i \vert_U}$ to $U$ converge in the $C^{\infty}$-topology
to $\overline{g}$ (actually $C^1$-convergence would be enough in the sequel). It follows, in particular, that
$\mu (\overline{g} (\calb))
= \lim_{i \rightarrow \infty} \mu (g_{i \vert_U} (\calb))$. Equation~(\ref{corollaryb}) and the
{\it uniform}\, convergence of ${\rm Jac}\, [Dg_{i \vert_U}]$ towards ${\rm Jac}\, [D\overline{g}]$
yields the estimate
$$
\mu (\overline{g} (\calb)) = \lim_{i \rightarrow \infty} \mu (g_{i \vert_U} (\calb))
\geq C \lim_{i \rightarrow \infty} \int_{\calb} \Vert {\rm Jac}\, [Dg_{i \vert_U}] \Vert^d d\mu = C
\int_{\calb} \Vert {\rm Jac}\, [D\overline{g}] \Vert^d d\mu \, .
$$
The last integral is clearly strictly positive since $\overline{g}$ is a diffeomorphism onto its image.
Hence $\mu (\overline{g} (\calb)) > 0$ and thus $\mu$ is quasi-invariant by elements in the $C^{\infty}$-closure
of $G$. In view of Theorem~A, we conclude that $\mu$ is absolutely continuous.
\end{proof}

Similarly, let us now prove Corollary~C.

\begin{proof}[Proof of Corollary~C]
The foliations $\mathcal{F}$ on $\C P(n+1)$ constructed in \cite{frankandI} are such that there exists a
local transverse section $\Sigma$ satisfying the following conditions
(cf. Theorem~9.3 in \cite{frankandI}):
\begin{enumerate}
  \item The section $\Sigma$ is isomorphic to a poly-disc in $\C^{n}$.
  \item Every (regular) leaf $L$ of $\mathcal{F}$ intersects $\Sigma$
  \item The pseudogroup $G_{\Sigma}$ of local diffeomorphisms of $\Sigma$ induced by the holonomy pseudogroup of $\mathcal{F}$
  on $\Sigma$ has ``large affine part'', i.e. in suitable coordinates its closure contains a copy of ${\rm SL}\, (n ,\C) \ltimes
  \C^n$.
\end{enumerate}
For the reasons already explained, every transverse $d$-quasi-volume for $\mathcal{F}$ must induce a finite (non-zero)
measure $\mu$ on $\Sigma$ that will be a $d$-quasi-volume for $G_{\Sigma}$. In particular, $\mu$ must be quasi-invariant
for the closure of $G_{\Sigma}$ and hence, in suitable coordinates, by the natural action
of ${\rm SL}\, (n ,\C) \ltimes \C^n$. Proposition~\ref{uniquenesslebesgue} then shows that $\mu$ must be absolutely
continuous. Corollary~C is proved.
\end{proof}

\section{The rigidity phenomenon}

Let us begin the discussion of the rigidity phenomenon associated to groups as above
by proving Theorem~D. Let $G_1, G_2$ be subgroups of $\diffM$ as in the statement of Theorem~D.
Denote by $H$ a continuous orbit equivalence between the actions of $G_1, \, G_2$. According to the main
result of \cite{david}, the homeomorphism $H$ is equivariant and hence it constitutes a topological
conjugacy between the actions in question. In fact, the only use made of condition~(6) in Section~2 is to allow us
to resort to the mentioned theorem of Fisher and Whyte.

Thus we have a topological conjugacy $H$ between the actions of $G_1, \, G_2$, in other words,
given $g \in G_2$, the homeomorphism $H^{-1} \circ g \circ H$ actually coincides with an element
of $G_1$. Theorem~D can then be rephrased as follows.

\begin{teo}
\label{toprigidity}
Let $G_1, G_2$ be as in the statement of Theorem~D. Every homeomorphism $H : M \rightarrow M$
conjugating the dynamics of $G_1$ and $G_2$ coincides with an element of $\diffM$.
\end{teo}

Yet, Theorem~\ref{toprigidity} bears a fundamental difference from the analogous ``topological
rigidity theorems'' proved, for example, in \cite{rebelo-ENS}, \cite{belliart-2} and \cite{rebelo-BBMS}. Namely,
the present statement makes no assumption on $H$ being ``close to the identity'' (though we assume that
$G_1$ acts ``linearly'' on $M$). More precisely, if $g_2 \in G_2$ is an element (say belonging to a fixed
generating set for $G_2$) that happens to be ``close to the identity'' (for a chosen topology), nothing ensures
that $g_1 \in G_1$ given by $g_1 = H^{-1} \circ g_2 \circ H$ is ``close to the identity'' as well.

The key point for all these rigidity theorems consists of constructing vector fields {\it conjugated by $H$}\,
(i.e. the vector fields are synchronized by $H$). Recall that two (local) analytic vector fields $X^1, X^2$ are
said to be conjugated by $H$ if the equation
\begin{equation}
H \circ \phi_{X^1}^t (x) = \phi_{X^2}^t (H(x)) \label{synchronization}
\end{equation}
is satisfied whenever both sides are defined, where $\phi_{X^1}^t, \phi_{X^2}^t$ denote the local flows
respectively associated to $X^1, X^2$. In fact, the fundamental result for the proof of Theorem~D is
the following proposition:

\begin{prop}
\label{propfortoprigidity}
Fix a point $q \in M$. Then there are analytic vector fields $X^1, X^2$ with $X^1$ defined about $q$ and
$X^2$ defined about $H(q)$ such that the conditions below are satisfied.
\begin{enumerate}

\item $X^1$ (resp. $X^2$) is contained in the $C^1$-closure of $G_1$ (resp. $G_2$).
Besides $X^1 (q) \neq 0$ (resp. $X^2 (H(q)) \neq 0$).

\item If $\phi_{X^1}^t$ (resp. $\phi_{X^2}^t$) denotes the local flow associated to $X^1$ (resp. $X^2$),
then Equation~(\ref{synchronization}) holds whenever both sides are defined.
\end{enumerate}
\end{prop}

Assuming Proposition~\ref{propfortoprigidity} holds, the proof of Theorem~D can be derived as follows.

\vspace{0.2cm}

\begin{proof}[Proof of Theorem~D]
Fix $q \in M$ and vector fields $X^1, X^2$ as in Proposition~\ref{propfortoprigidity}. In view of the discussion
conducted in Section~4.3, the group $G_1$ acts on $X^1$ to yield a collection of vector fields $X^1 = X^1_1,
\ldots , X_n^1$ contained in the closure of $G$ and such that the vectors $X_1^1 (q), \ldots ,X^1_n (q)$ form
a basis for $T_qM$. However, to every element $g_1 \in G_1$ acting on $X^1$ there corresponds an element $g_2
= H \circ g_1 \circ H^{-1} \in G_2$ acting on $X^2$. It then follows the existence of vector fields
$X^2 =X^2_1, \ldots ,X_n^2$ fulfilling the following conditions:
\begin{itemize}

\item $X^2_1, \ldots ,X_n^2$ are analytic vector fields (contained in the $C^1$-closure of $G$).

\item $X^2_1 (H(q)), \ldots ,X_n^2 (H(q))$ form a basis for $T_{H(q)} M$.

\item For every $i \in \{ 1, \ldots ,n\}$, $H$ conjugates the vector fields $X^1_i$ and $X^2_i$.
\end{itemize}

Consider now the local analytic coordinates $(s_1, \ldots ,s_n)$ (resp. $(\overline{s}_1, \ldots,
\overline{s}_n)$) about $q$ (resp. $H(q)$) defined by
$$
(s_1, \ldots ,s_n) \longmapsto \phi_{X^1_n}^{s_n} \circ \cdots \circ \phi_{X^1_1}^{s_1}
$$
(resp. $(\overline{s}_1, \ldots, \overline{s}_n) \mapsto \phi_{X^2_n}^{\overline{s}_n} \circ \cdots \circ
\phi_{X^2_1}^{\overline{s}_1}$). Since $H$ conjugates the vector fields $X^1_i$ and $X^2_i$ for
every $i \in \{ 1, \ldots ,n\}$, it becomes clear that $H$, mapping $q$ to $H(q)$, locally coincides
with the identity in the mentioned coordinates. Thus $H$ is analytic about $q$. Since $q$ is an arbitrary
point of $M$, the statement of Theorem~D follows at once.
\end{proof}

Let us now supply the proof of Proposition~\ref{propfortoprigidity} by exploiting
the ``linear'' character of the action of $G_1$ on $M$.

\begin{proof}[Proof of Proposition~\ref{propfortoprigidity}]
We shall mainly work with the ``non-linear'' group $G_2$. As always, it suffices
to prove the statement on a neighborhood of a contracting fixed point $p$ of some Morse-Smale element of $G_2$.

For this, note that the proof of the proposition in question requires us to construct two conjugate analytic
vector fields: whether or not they belong to the closure of the corresponding groups $G_1, \, G_2$ is no longer
relevant. Thus we shall base our construction in Proposition~\ref{vectorfield3} yielding analytic vector fields
in the $C^r$-closure of $G_2$ (actually we may work with $r=1$). Also it suffices to find $X^1, X^2$ such that
$X^2$ is not identically zero (and hence neither is $X^1$). The condition $X^2 (p) \neq 0$ can immediately
be achieved by replacing the ``original'' $X^2$ by $g^{\ast} X^2$ for a suitable $g \in G_2$ (recall that
$G_2$ has dense orbits in $M$).

The construction of $X^1, X^2$ goes as follows. Let $\{ h_i^2\}$ be a sequence of elements in $G_2$
such that the vector fields $X_i^2$ defined by
\begin{equation}
X_i^2 (x) = \frac{1}{\sup_{x \in U} \Vert h_i^2 (x) -x \Vert} . (h_i^2 (x) -x) \label{synchronizing1}
\end{equation}
converge in the $C^r$-topology to a non-identically zero analytic vector field $X^2$ in the
$C^r$-closure of $G_2$.

Consider now the sequence $\{ h_i^1\}$ of elements in $G_1$ given by $h_i^1 = H^{-1} \circ h_i^2
\circ H$. Clearly the sequence $\{ h_i^1\}$ converges {\it uniformly}\, to the identity on $M$. However,
since $G_1$ is ``linear'', this sequence also converges to the identity in, say, the $C^{\infty}$-topology
(it is an easy consequence of Lemma~\ref{vectorfield4}). Besides, this sequence also satisfies the assumptions
of Proposition~\ref{vectorfield1}, cf. again Lemma~\ref{vectorfield4}). Thus, letting
\begin{equation}
X_i^1 (x) = \frac{1}{\sup_{x \in U} \Vert h_i^1 (x) -x \Vert} . (h_i^1 (x) -x) \, .\label{synchronizing2}
\end{equation}
Modulo passing to a subsequence, the vector fields $X_i^1$ will converge in the $C^r$-topology towards an
analytic vector field $X^1$ in the $C^r$-closure of $G_1$. Since, for every $i \in \N$,
$h_i^1 = H^{-1} \circ h_i^2 \circ H$, it follows that $H$ conjugates $X^1$ to $X^2$. The proposition
is proved.
\end{proof}

\subsection{Additional issues and questions on OE}

As promised, this paper will end with a few questions concerning the ``rigidity
phenomenon'' and orbit equivalences.
Let us start with a general question concerning the role of the ``linear'' assumption on $G_1$ imposed by
the statement of Theorem~D. Note that this assumption was only used in the proof of Proposition~\ref{propfortoprigidity}.
In most situations, the sequence $\{ g_j^2\}_{j \in \N}$ of elements in $G^2$ converging to the
identity and satisfying the conditions of Lemma~\ref{vectorfield1} can be obtained by re-normalizing (by means of
$F$) a sequence of commutators. In these cases, an affirmative answer to the question below would allow us
to remove the ``linear'' assumption from the corresponding statement (for details see the standard construction
for example in \cite{rebelo-ENS} or \cite{rebelo-BBMS}).

\vspace{0.1cm}

\noindent {\it Question: does $C^0$-convergence imply $C^{\infty}$-convergence}. To accurately state this question, let $f, g$ be two diffeomorphisms of $M$. By taking $\{ f,g, f^{-1}, g^{-1} \}$
as the generating set $S$ of a group $G$, consider the sequence of subsets $S(j)$ defined by Ghys, cf. Section~2.
We may assume that $f,g$ are close to the identity so that every non-trivial sequence of diffeomorphisms
$\{ g_k\}_{k \in \N}$ such that, say, $G_k \in S(k)$, converges to the identity (say in the $C^{\infty}$-topology).

Now let $H$ denote a homeomorphism of $M$ such that $\overline{f} = H^{-1} \circ f \circ H$ and
$\overline{g} = H^{-1} \circ g \circ H$ are smooth/analytic diffeomorphism of $M$. Let $\overline{S} (j)$ be the
analogous sequence of commutators that is obtained by starting with $\overline{S} = \{ {f},
\overline{g} , \overline{f}^{-1} , \overline{g}^{-1} \}$. In particular, for every $k \in \N$ and $g_k \in S(k)$,
there corresponds a diffeomorphism $\overline{g}_k \in \overline{S} (k)$. Clearly the sequence of diffeomorphisms
$\{ \overline{g}_k \}$ converges in the $C^0$-topology to the identity. In view of what precedes, it is natural
to ask whether or not $\{ \overline{g}_k \}$ converges to the identity in the $C^{\infty}$-topology as well? This
question appears to be interesting already when the ambient manifold $M$ is the circle. Note, however, that the
problem is immaterial in the holomorphic setting (even for pseudogroups) due to the Cauchy formula.

\vspace{0.1cm}

\noindent {\it Rigidity for measurable conjugacy}. Let $G_1, G_2 \subset \diffM$ be as in the statement of
Theorem~D. Assume that $\theta : M \rightarrow M$ is a measurable bijection conjugating (Lebesgue almost
everywhere) the actions of $G_1, G_2$. Our purpose here is to sketch a proof that, possibly modulo some
minor additional generic assumptions, $\theta$ must coincide almost everywhere with an element of $\diffM$.
The argument is borrowed from \cite{reb2}. First consider the diagonal action $G_1 \times G_2$ on
$M \times M$ and note that vector fields in the $C^{\infty}$-closure of this action can be constructed. In fact,
it suffices to consider $X^2$ as in the proof of Proposition~\ref{propfortoprigidity} and obtain a vector field
$X^1$ by means of the sequence $h_i^1 = \theta^{-1} \circ h_i^2 \circ \theta$. Here the classical Lusin theorem
applied to $\theta$ suffices to conclude that $\{ h_i^1\}$ converges $C^0$, and hence $C^{\infty}$, to the
identity. The resulting vector field $X = (X^1, X^2)$ is then in the $C^{\infty}$-closure of $G_1 \times G_2$.
Now note that the graph ${\rm Graph}\, (\theta)$ of $\theta$ in $M \times M$ is invariant under this diagonal
action. The same argument given in Section~4 of \cite{reb2} shows the following: a generic (Lebesgue density)
point $p_1$ in the domain of $\theta$ is such that no vector field $X =(X^1,X^2)$ in the $C^{\infty}$-closure
of the diagonal action and defined about $(p_1, \theta (p_1))$ is neither {\it vertical}\, nor {\it horizontal}\,
at $(p_1, \theta (p_1))$. By definition this means that every vector field $X = (X^1, X^2)$ in the
$C^{\infty}$-closure of $G_1 \times G_2$ and defined about $(p_1, \theta (p_1))$ is such that either
$X^1 (p_1), X^2 (\theta (p_1))$ vanish simultaneously or they are both different from zero. The third step
of the proof, for which generic assumptions may be required, consists of showing that the set of points
$(p_1,p_2) \in M \times M$ such that no vector field defined about $(p_1,p_2)$ and contained in the
$C^{\infty}$-closure of $G_1 \times G_2$ is either vertical or horizontal at $(p_1,p_2)$ constitutes an
analytic subset of dimension~$n$ of $M \times M$. Once this last statement is proved,
it follows that ${\rm Graph}\, (\theta)$ must be
contained in the analytic set in question so that the regular character of $\theta$ can easily be derived.

\vspace{0.1cm}

\noindent {\it OE and further questions}. Our main interest in the above mentioned rigidity result lies
in the fact that {\it it is easy to construct actions that are not $C^{\infty}$-conjugate}\, and hence
{\it it becomes easy}\, to construct actions that are not measurably conjugate. A general question is then
as follows: can we use this result to construct new examples of non-OE actions for a same abstract group?

There are many difficult questions concerning OE between actions (see for example \cite{furman}) despite
some rather spectacular progress made in the past decade or so. Yet, constructing examples of non-OE actions
for a same group with as much regularity as possible is still an interesting question. If the ``very regular''
actions for which the ``rigidity of measurable conjugacies'' holds are considered, we may wonder when it is
possible to ensure that a given OE must necessarily be equivariant. Already in the case of subgroups
of ${\rm PSL}\, (2, \C)$ projectively acting on the sphere $S^2$, some specific question can explicitly
be stated, namely:
\begin{itemize}
  \item According to Sullivan \cite{sullivan}, every Kleinian group is OE to a $\Z$-action and hence the action
  of every two Kleinian groups are OE (in the context of quasi-invariant measures). The proof depends heavily
  on the fact that these groups are discrete. So we may wonder how many non-OE actions can be produced by means
  of, say the free group on two generators, realized as dense subgroups of ${\rm PSL}\, (2, \C)$.

  \item For which classes of groups, if any, admitting non-discrete actions every OE must be equivariant?

  \item Related to the last question above, is there a generalization of the main result in \cite{david}
  to the measure-theoretic context? Whereas this extension seems unlikely to exist, almost every non-trivial
  generalization of the criterion in \cite{david} would be very interesting.

\end{itemize}

\begin{flushleft}

{\sc Julio C. Rebelo}\\
Institut de Math\'ematiques de Toulouse\\
Universit\'e de Toulouse\\
118 Route de Narbonne F-31062, Toulouse\\
FRANCE\\
rebelo@math.univ-toulouse.fr

\end{flushleft}

\end{document}